\documentclass[11pt,a4paper]{amsart}
\usepackage[latin9]{inputenc}
\usepackage{mathtools}
\usepackage{amscd}
\usepackage{amstext}
\usepackage{amsthm}
\usepackage{amssymb}
\usepackage[unicode=true,
 bookmarks=false,
 breaklinks=false,pdfborder={0 0 1},backref=false,colorlinks=true,linkcolor=blue,citecolor=blue]
 {hyperref}
\usepackage{cleveref}
\usepackage{soul}
\usepackage{xcolor}
\usepackage{comment}

\makeatletter

\newcommand*\LyXbar{\rule[0.585ex]{1.2em}{0.25pt}}
\pdfpageheight\paperheight
\pdfpagewidth\paperwidth

\numberwithin{equation}{section}
\numberwithin{figure}{section}
\theoremstyle{plain}
\newtheorem{thm}{\protect\theoremname}[section]
  \theoremstyle{definition}
  
 \theoremstyle{definition}

  \theoremstyle{definition}
  \newtheorem{defn}[thm]{\protect\definitionname}
 \theoremstyle{definition}

  \theoremstyle{definition}
  \newtheorem{example}[thm]{\protect\examplename}
  \theoremstyle{definition}
  \newtheorem{notation}[thm]{\protect{Notation}}
  \theoremstyle{plain}
  \newtheorem{prop}[thm]{\protect\propositionname}
  \theoremstyle{plain}
  \newtheorem{cor}[thm]{\protect\corollaryname}
  \theoremstyle{plain}
  \newtheorem{lem}[thm]{\protect\lemmaname}
 \theoremstyle{definition}
 \newtheorem{rem}[thm]{\protect\remarkname}

\@ifundefined{date}{}{\date{}}

\usepackage{tikz}
\usepackage{amsthm}
\usepackage[hang]{footmisc}
\usepackage[all]{xypic}
\usepackage{color}
\usepackage[inline]{enumitem}
\usepackage{mathrsfs}

\allowdisplaybreaks

\newtheorem{innercustomthm}{Theorem}\newenvironment{customthm}[1]{\renewcommand\theinnercustomthm{#1}\innercustomthm}{\endinnercustomthm}

\makeatother

  \providecommand{\corollaryname}{Corollary}
  \providecommand{\definitionname}{Definition}
  \providecommand{\examplename}{Example}
  \providecommand{\lemmaname}{Lemma}
  \providecommand{\problemname}{Problem}
  \providecommand{\propositionname}{Proposition}
 \providecommand{\remarkname}{Remark}
\providecommand{\theoremname}{Theorem}

\begin{document}

\global\long\def\floorstar#1{\lfloor#1\rfloor}
\global\long\def\ceilstar#1{\lceil#1\rceil}

\global\long\def\B{B}
 \global\long\def\A{A}
 \global\long\def\J{J}
 \global\long\def\K{\mathcal{K}}
 \global\long\def\D{D}
 \global\long\def\Ch{D}
 \global\long\def\Zh{\mathcal{Z}}
 \global\long\def\E{E}
 \global\long\def\Oh{\mathcal{O}}

\global\long\def\T{{\mathbb{T}}}
 \global\long\def\BR{{\mathbb{R}}}
 \global\long\def\N{{\mathbb{N}}}
 \global\long\def\Z{{\mathbb{Z}}}
 \global\long\def\C{{\mathbb{C}}}
 \global\long\def\Q{{\mathbb{Q}}}

\global\long\def\aut{\mathrm{Aut}}
 \global\long\def\supp{\mathrm{supp}}

\global\long\def\eps{\varepsilon}

\global\long\def\id{\mathrm{id}}

\global\long\def\halpha{\widehat{\alpha}}
 \global\long\def\calpha{\widehat{\alpha}}

\global\long\def\tih{\widetilde{h}}

\global\long\def\opFol{\operatorname{F{\o}l}}

\global\long\def\opRange{\operatorname{Range}}

\global\long\def\opIso{\operatorname{Iso}}

\global\long\def\dimnuc{\dim_{\operatorname{nuc}}}

\global\long\def\set#1{\left\{  #1\right\}  }

\global\long\def\mset#1{\left\{  \!\!\left\{  #1\right\}  \!\!\right\}  }

\global\long\def\IA{\mathbb{A}}
 \global\long\def\IB{\mathbb{B}}
 \global\long\def\IC{\mathbb{C}}
 \global\long\def\ID{\mathbb{D}}
 \global\long\def\IE{\mathbb{E}}
 \global\long\def\IF{\mathbb{F}}
 \global\long\def\IG{\mathbb{G}}
 \global\long\def\IH{\mathbb{H}}
 \global\long\def\II{\mathbb{I}}
 \global\long\def\IJ{\mathbb{J}}
 \global\long\def\IK{\mathbb{K}}
 \global\long\def\IL{\mathbb{L}}
 \global\long\def\IM{\mathbb{M}}
 \global\long\def\IN{\mathbb{N}}
 \global\long\def\IO{\mathbb{O}}
 \global\long\def\IP{\mathbb{P}}
 \global\long\def\IQ{\mathbb{Q}}
 \global\long\def\IR{\mathbb{R}}
 \global\long\def\IS{\mathbb{S}}
 \global\long\def\IT{\mathbb{T}}
 \global\long\def\IU{\mathbb{U}}
 \global\long\def\IV{\mathbb{V}}
 \global\long\def\IW{\mathbb{W}}
 \global\long\def\IX{\mathbb{X}}
 \global\long\def\IY{\mathbb{Y}}
 \global\long\def\IZ{\mathbb{Z}}

\global\long\def\CA{\mathcal{A}}
 \global\long\def\CB{\mathcal{B}}
 \global\long\def\CC{\mathcal{C}}
 \global\long\def\CalD{\mathcal{D}}
 \global\long\def\CD{\mathcal{D}}
 \global\long\def\CE{\mathcal{E}}
 \global\long\def\CF{\mathcal{F}}
 \global\long\def\CG{\mathcal{G}}
 \global\long\def\CH{\mathcal{H}}
 \global\long\def\CI{\mathcal{I}}
 \global\long\def\CJ{\mathcal{J}}
 \global\long\def\CK{\mathcal{K}}
 \global\long\def\CL{\mathcal{L}}
 \global\long\def\CM{\mathcal{M}}
 \global\long\def\CN{\mathcal{N}}
 \global\long\def\CO{\mathcal{O}}
 \global\long\def\CP{\mathcal{P}}
 \global\long\def\CQ{\mathcal{Q}}
 \global\long\def\CR{\mathcal{R}}
 \global\long\def\CS{\mathcal{S}}
 \global\long\def\CT{\mathcal{T}}
 \global\long\def\CU{\mathcal{U}}
 \global\long\def\CV{\mathcal{V}}
 \global\long\def\CW{\mathcal{W}}
 \global\long\def\CX{\mathcal{X}}
 \global\long\def\CY{\mathcal{Y}}
 \global\long\def\CZ{\mathcal{Z}}

\global\long\def\FA{\mathfrak{A}}
 \global\long\def\FB{\mathfrak{B}}
 \global\long\def\FC{\mathfrak{C}}
 \global\long\def\FD{\mathfrak{D}}
 \global\long\def\FE{\mathfrak{E}}
 \global\long\def\FF{\mathfrak{F}}
 \global\long\def\FG{\mathfrak{G}}
 \global\long\def\FH{\mathfrak{H}}
 \global\long\def\FI{\mathfrak{I}}
 \global\long\def\FJ{\mathfrak{J}}
 \global\long\def\FK{\mathfrak{K}}
 \global\long\def\FL{\mathfrak{L}}
 \global\long\def\FM{\mathfrak{M}}
 \global\long\def\FN{\mathfrak{N}}
 \global\long\def\FO{\mathfrak{O}}
 \global\long\def\FP{\mathfrak{P}}
 \global\long\def\FQ{\mathfrak{Q}}
 \global\long\def\FR{\mathfrak{R}}
 \global\long\def\FS{\mathfrak{S}}
 \global\long\def\FT{\mathfrak{T}}
 \global\long\def\FU{\mathfrak{U}}
 \global\long\def\FV{\mathfrak{V}}
 \global\long\def\FW{\mathfrak{W}}
 \global\long\def\FX{\mathfrak{X}}
 \global\long\def\FY{\mathfrak{Y}}
 \global\long\def\FZ{\mathfrak{Z}}

\global\long\def\Ra{\Rightarrow}
 \global\long\def\La{\Leftarrow}
 \global\long\def\LRa{\Leftrightarrow}

\global\long\def\quer{\overline{}}
 \global\long\def\eins{\mathbf{1}}
 \global\long\def\diag{\operatorname{diag}}
 \global\long\def\ad{\operatorname{Ad}}
 \global\long\def\ev{\operatorname{ev}}
 \global\long\def\fin{{\subset\!\!\!\subset}}
 \global\long\def\diam{\operatorname{diam}}
 \global\long\def\Hom{\operatorname{Hom}}
 \global\long\def\dst{{\displaystyle }}
 \global\long\def\spp{\operatorname{supp}}
 \global\long\def\spo{\operatorname{supp}_{o}}
 \global\long\def\del{\partial}
 \global\long\def\lsc{\operatorname{Lsc}}
 \global\long\def\GU{\CG^{(0)}}
 \global\long\def\HU{\CH^{(0)}}
 \global\long\def\AU{\CA^{(0)}}
 \global\long\def\BU{\CB^{(0)}}
 \global\long\def\CUU{\CC^{(0)}}
 \global\long\def\DU{\CD^{(0)}}
 \global\long\def\CUUU{\CC'{}^{(0)}}

\global\long\def\AUl{(\CA^{l})^{(0)}}
\global\long\def\BUl{(B^{l})^{(0)}}
\global\long\def\HUp{(\CH^{p})^{(0)}}

\global\long\def\properlength{proper}

\global\long\def\interior#1{#1^{\operatorname{o}}}


\title[Fiberwise amenability of \'{e}tale groupoids]{Fiberwise amenability of \'{e}tale groupoids
}

\author{Xin~Ma}

\address{X.~Ma: Institute for Advanced Study in Mathematics, Harbin Institute of Technology, Harbin, 150001, China}

\email{xma17@hit.edu.cn}

\author{Jianchao~Wu}

\address{J.~Wu: Shanghai Center for Mathematical Sciences, Fudan University,
	Shanghai 200438, China}

\email{jianchao\_wu@fudan.edu.cn}

\keywords{Coarse geometry, Fiberwise amenability, \'{E}tale groupoids}
\subjclass[2000]{22A22, 46L35, 51F30, 37A55, 37B05}
\begin{abstract}
We introduce a new amenability property for \'{e}tale groupoids, termed \textit{fiberwise amenability}, along with a stronger variant termed \emph{ubiquitous fiberwise amenability}.  (Ubiquitous) fiberwise amenability emerges naturally from a coarse-geometric perspective on \'{e}tale groupoids and, in the special case of transformation groupoids, it coincides precisely with the amenability of the acting group (rather than topological amenability of the action). It is also tightly linked to the existence of invariant measures on the unit space of the groupoid. The coarse-geometric framework for \'{e}tale groupoids that we develop systematically in this work allows us to establish several foundational properties of (ubiquitous) fiberwise amenability. 
As an application, we prove a F\o lner--paradoxical dichotomy for minimal \'{e}tale groupoids, which will serve as a key tool in a sequel on almost elementariness of \'{e}tale groupoids.
\end{abstract}

\maketitle
\tableofcontents{}

\section{Introduction}

Amenability originated in group theory as a property reflecting the interplay between rigidity and flexibility. Its foundations lie in von Neumann's formulation in terms of invariant means, motivated by the Banach--Tarski paradox \cite{BanachTarski1924,vonNeumann1929}, and in F\o{}lner's characterization by finite sets with asymptotically small boundaries \cite{Folner1955}. It gradually became apparent, however, that many phenomena traditionally formulated in terms of amenable groups are more naturally understood in the broader contexts of group actions, measured equivalence relations, and geometric structures. This perspective led to several extensions of amenability beyond groups themselves, most notably to group actions, equivalence relations, and groupoids \cite{Zimmer1978,ConnesFeldmanWeiss1981,Renault1980groupoid}.

In the topological setting, locally compact groupoids provide the natural framework for such extensions. Renault's treatment of amenability for measured groupoids, followed by the development of \emph{topological amenability} for locally compact groupoids, established a powerful and unifying language encompassing amenable groups, amenable group actions, and amenable equivalence relations \cite{Renault1980groupoid,ADR2000}. Meanwhile, in coarse geometry, Yu introduced for metric spaces an amenability-type notion termed \emph{property~A} \cite{Yu2000}, with major applications to the coarse Baum-Connes conjecture and the strong Novikov conjecture. It was then realized that through the construction of the coarse groupoid of a metric space of bounded geometry, property~A corresponds precisely to topological amenability \cite{SkandalisTuYu2002coarse}. This notion has since become fundamental at the interface between operator algebras and coarse geometry as well as geometric group theory, particularly through its relationships with the nuclearity and exactness of uniform Roe algebras and reduced group C*-algebras \cite{Yu2000,Roe2003Lectures,Sako2020}.

On the other hand, there is another coarse-geometric amenability-type notion, namely \emph{metric amenability} \cite{BlockWeinberger1992Aperiodic}, which directly generalizes the F\o{}lner-set characterization of amenability for groups, and behaves quite differently from property~A: 
unlike the latter, metric amenability relates to tracial states on uniform Roe algebras \cite{AraLiLledoWu2018Amenabilitya}, instead of nuclearity or exactness. 
This distinction between the two amenability-type notions is akin to how topological amenability of a group action on a compact space does not reflect amenability of the acting group, nor does it guarantee the existence of an invariant probability measure. The action $\mathbb{F}_2\curvearrowright \del\mathbb{F}_2$ of a free group on its Gromov boundary illustrates this phenomenon: although $\mathbb{F}_2$ is nonamenable, its boundary action, and hence the associated transformation groupoid, is topologically amenable, while $\del\mathbb{F}_2$ admits no $\mathbb{F}_2$-invariant probability measure \cite{Adams1994,ADR2000}.

These considerations motivate the introduction of \emph{fiberwise amenability} for \'{e}tale groupoids (see Definition~\ref{def:fiberwise-amenable}), a notion designed to recover the amenability of acting groups for transformation groupoids and metric amenability for coarse groupoids, while complementing the global notion of topological amenability. The adjective fiberwise reflects the fact that our notion is designed to capture the amenability intrinsic to the large-scale geometry of a groupoid's fibers. This is natural in the above examples: a coarse metric space can be recovered from the fibers of its coarse groupoid, and the coarse geometry of the acting group of an action is retained by the fibers of the transformation groupoid. 

Let us summarize the main features of fiberwise amenability. 
Most notably, when the unit space of a groupoid is compact, it dynamically guarantees the existence of invariant probability measures (see Theorem~\ref{thm: B}). For a transformation groupoid $X\rtimes\Gamma$, it reflects the amenability of the acting group $\Gamma$ (see Remark~\ref{rem:transformation-groupoid-fiberwise-amenable}). For the coarse groupoid associated with a metric space, it recovers metric amenability of the underlying coarse space (see Proposition~\ref{prop:coarse-groupoid-amenable}). Fiberwise amenability therefore differs from topological amenability of \'{e}tale groupoids: for transformation groupoids, the latter corresponds to topological amenability of the action, whereas for coarse groupoids, it corresponds to Yu's property~A \cite{ADR2000,SkandalisTuYu2002coarse}. Although amenability of the acting group implies topological amenability of the action \cite{ADR2000}, fiberwise amenability is, in general, neither stronger nor weaker than topological amenability, as demonstrated by coarse groupoids (see \Cref{rem:fa-vs-ta}). More striking examples arise from the minimal ample \'{e}tale groupoids introduced by Elek in \cite{Elek-qualitative}, which are fiberwise amenable but not topologically amenable (see \Cref{exa: Elek}).

A key technical ingredient underlying our approach is a systematic method for associating a canonical coarse metric structure with any $\sigma$-compact \'{e}tale groupoid. More precisely, we show that every such groupoid admits a proper, continuous length function that is unique up to coarse equivalence (see Theorem~\ref{thm: A}). This construction extends the classical passage from a countable discrete group to a proper left-invariant metric \cite{Struble1974,Roe2003Lectures} and, more generally, places \'{e}tale groupoids firmly within the framework of coarse geometry \cite{SkandalisTuYu2002coarse}. Once this coarse structure is in place, fiberwise amenability arises naturally from metric amenability of the associated extended coarse metric space. This perspective not only clarifies the definition but also allows one to import intuition and techniques from coarse geometry into the groupoid setting.

Although fiberwise amenability captures many essential features, it is often too weak for applications requiring robust permanence properties. As in the metric setting, amenability may reflect only the behavior of part of a space or groupoid rather than its global structure. To address this issue, we introduce the stronger notion of \emph{ubiquitous fiberwise amenability} in Definition~\ref{def:fiberwise-amenable}. This condition requires amenability to be present uniformly and ubiquitously throughout the groupoid. 

Ubiquitous fiberwise amenability plays a decisive structural role for \'{e}tale and, in particular, ample groupoids. It provides a natural dividing line between two important regularity properties introduced by Matui: \emph{almost finiteness} and \emph{pure infiniteness} \cite{Matui2012Homology, Matui2015Topological}. This dichotomy closely parallels the division of unital classifiable $\mathrm{C}^*$-algebras into stably finite and purely infinite classes. 
One motivation for the present work is the observation that Matui's almost finiteness implies ubiquitous fiberwise amenability (see \Cref{prop:af-ufa}). This insight leads to a broader regularity property for \'{e}tale groupoids, called \emph{almost-elementariness}, which unifies almost finiteness and pure infiniteness without assuming ampleness. The development and applications of almost-elementariness will be taken up in a sequel to this paper \cite{MWAE}.

The main results of this paper can be summarized as follows. First, we establish the existence and uniqueness, up to coarse equivalence, of a proper continuous length function on every $\sigma$-compact \'{e}tale groupoid, thereby endowing it with a canonical coarse metric structure.

\begin{customthm}{A}[Theorem \ref{thm:coarse-length-functions}]\label{thm: A}
Let $\CG$ be a $\sigma$-compact,
\'{e}tale groupoid. 
Then, up to coarse equivalence, it has a unique coarse continuous length function. 
\end{customthm}

We introduce the fiberwise amenability and the ubiquitous fiberwise amenability, and prove that the fiberwise amenability implies  that there is always an invariant probability measure on the unit space. 

\begin{customthm}{B}[Corollary \ref{cor:fiberwise-amenable-invariant-measure}]\label{thm: B}
Let $\CG$ be a fiberwise amenable, $\sigma$-compact,
\'{e}tale groupoid with a compact unit space.
Then the set $M(\CG)$ of all $\CG$-invariant probability measures is not empty. 
\end{customthm}

Moreover, we show that for minimal \'{e}tale groupoids these two notions coincide.  The proof of this equivalence relies on a local slice lemma (see Lemma \ref{lem:local-slice}), which allows one to start from a single F\o{}lner set and make near-identical copies in a controlled way.  

\begin{customthm}{C}[Theorem \ref{thm:minimal-fiberwise-amenability}]\label{thm: C}
    Let $\CG$
be a $\sigma$-compact \'{e}tale groupoid. Suppose
$\CG$ is minimal. Then $\CG$ is fiberwise amenable if and only
if it is ubiquitously fiberwise amenable.
\end{customthm}

Furthermore, we show that ubiquitous metric (or fiberwise) amenability yields a F\o{}lner--versus--paradoxical type result.

\begin{customthm}{D}[Theorem \ref{thm:fiberwise-amenable-dichotomy}]\label{thm: D}
    Let $\CG$ be a 
$\sigma$-compact \'{e}tale groupoid with a compact unit space. 
Then we have the following. 
\begin{enumerate}
\item If $\CG$ is ubiquitous fiberwise amenable, then 
for any compact subset $K$ in $\CG$ and any $\varepsilon>0$, 
there is a compact set $L \subseteq \CG$ such that
for any finite set $M \subseteq\CG$, there
is a finite set $F$ satisfying 
\[
	M \subseteq F\subseteq LM \ \textrm{and}\ |KF|\leq(1+\varepsilon)|F|.
\]

\item If $\CG$ is not fiberwise amenable, then 
for any compact set $K \subseteq\CG$
and any $n\in\mathbb{N}$, there is a compact set $L \subseteq\CG$ such that
for any finite set $M\subseteq\CG$, the set
$LM$ contains at least $n|M|$ many disjoint sets of the form $K x$ for $x \in \CG$. 
\end{enumerate}
\end{customthm}
In the minimal case, Theorem \ref{thm: D}, combined with Theorem \ref{thm: C}, provides a genuine dichotomy, which will be instrumental in our later study of almost-elementariness and groupoid strict comparison in \cite{MWAE}.

The paper is organized as follows.  In Section~\ref{sec:preliminaries}, we collect background material on \'{e}tale groupoids, invariant measures, and coarse geometry, and fix notation used throughout the paper.  Section~\ref{sec:metric-amenability} studies amenability for extended coarse metric spaces, revisiting metric amenability before introducing its ubiquitous variant.  In Section~\ref{sec:coarse-geometry}, we develop the coarse geometry of \'{e}tale groupoids, including the construction and uniqueness of proper continuous length functions.  Finally, Section~\ref{sec:fiberwise-amenability} introduces fiberwise amenability and ubiquitous fiberwise amenability for \'{e}tale groupoids, establishes their basic properties, proves the main equivalence and dichotomy results, and investigates a number of special cases of interest.

\section{Preliminaries\label{sec:preliminaries}}

In this section we recall some basic backgrounds on coarse geometry and
\'{e}tale groupoids.

In this paper, there are two types of metric spaces under consideration.
One concerns usual topological metrizable spaces focus on local behavior
while another are coarse metric spaces from large scale geometric
point of view. However, even these two types have different nature,
as metric spaces, they share some same notations. Let $(X,d)$ be
a metric space equipped with the metric $d$. We denote by $B_{d}(x,R)$
the open ball $B_{d}(x,R)=\{y\in X:d(x,y)<R\}$ and by $\bar{B}_{d}(x,R)$
the closed ball $\bar{B}_{d}(x,R)=\{y\in X:d(x,y)\leq R\}$. Let $A$
be a subset of $X$. We write $B_{d}(A,R)$ and $\bar{B}_{d}(A,R)$
for analogous meaning. If the metric is clear, we write $B(A,R)$
and $\bar{B}(A,R)$ instead for simplification. We refer to \cite{NowakYu2012Large}
as a standard reference for topics of large scale geometry. 

We refer to \cite{Renault1980groupoid} and \cite{Sims-groupoids}
as references for groupoids and we record several fundamental definitions
and results for locally compact Hausdorff \'{e}tale groupoids here. 
\begin{defn}
A \textit{groupoid} $\CG$ is a set equipped with a distinguished
subset $\CG^{(2)}\subset\CG\times\CG$, called the set of \textit{composable
pairs}, a product map $\CG^{(2)}\rightarrow\CG$, denoted by $(\gamma,\eta)\mapsto\gamma\eta$
and an inverse map $\CG\rightarrow\CG$, denoted by $\gamma\mapsto\gamma^{-1}$
such that the following hold 
\begin{enumerate}[label=(\roman*)]
\item If $(\alpha,\beta)\in\CG^{(2)}$ and $(\beta,\gamma)\in\CG^{(2)}$
then so are $(\alpha\beta,\gamma)$ and $(\alpha,\beta\gamma)$. In
addition, $(\alpha\beta)\gamma=\alpha(\beta\gamma)$ holds in $\CG$.
\item For all $\alpha\in\CG$ one has $(\gamma,\gamma^{-1})\in\CG^{(2)}$
and $(\gamma^{-1})^{-1}=\gamma$.
\item For any $(\alpha,\beta)\in\CG^{(2)}$ one has $\alpha^{-1}(\alpha\beta)=\beta$
and $(\alpha\beta)\beta^{-1}=\alpha$. 
\end{enumerate}
Every groupoid is equipped with a subset $\GU=\{\gamma\gamma^{-1}:\gamma\in\CG\}$
of $\CG$. We refer to elements of $\GU$ as \textit{units} and to
$\GU$ itself as the \textit{unit space}. We define two maps $s,r:\CG\rightarrow\GU$
by $s(\gamma)=\gamma^{-1}\gamma$ and $r(\gamma)=\gamma\gamma^{-1}$,
respectively, in which $s$ is called the \textit{source} map and
$r$ is called the \textit{range} map. 
\end{defn}
\label{paragraph:sections-in-groupoids}
When a groupoid $\CG$ is endowed with a locally compact Hausdorff
topology under which the product and inverse maps are continuous,
the groupoid $\CG$ is called a locally compact Hausdorff groupoid.
A locally compact Hausdorff groupoid $\CG$ is called \textit{\'{e}tale}
if the range map $r$ is a local homeomorphism from $\CG$ to itself,
which means for any $\gamma\in\CG$ there is an open neighborhood
$U$ of $\gamma$ such that $r(U)$ is open and $r|_{U}$ is a homeomorphism.
In this case, since the map of taking inverses is an involutive homeomorphism
on $\CG$ that intertwines $r$ and $s$, thus the source map $s$
is also a local homeomorphism. 
A subset $S$ in $\CG$ is called an \emph{$s$-section} (respectively, an \emph{$r$-section})
if there is an open neighborhood $U$ of $S$ in $\CG$ such that
the source map $s$ (respectively,
the range map $r$) restricts to a homeomorphism from $U$
onto an open subset of $\GU$. 
It is called a \emph{bisection} if it is both an $s$-section and an $r$-section. 
It is not hard to see a locally compact
Hausdorff groupoid is \'{e}tale if and only if its topology has a basis
consisting of open bisections. We say a locally compact Hausdorff
\'{e}tale groupoid $\CG$ is \textit{ample} if its topology has a basis
consisting of compact open bisections.

\begin{example}\label{exa:pair-groupoid}
	Let $X$ be a discrete space. The \emph{pair groupoid} over $X$ has the discrete space $X \times X$ as the underlying topological space, and the groupoid operations are defined so that for any $u,v,w \in X$, we have 
	\[
		r(u,v) = u, \quad s(u,v) = v, \quad (u,v)^{-1} = (v,u) \quad \text{ and } \quad (u,v)(v,w) = (u,w) \; .
	\]
\end{example}

\begin{example}
\label{exa:transformation-groupoid}
Let $X$ be a locally compact
Hausdorff space and $\Gamma$ be a discrete group. Then any action
$\Gamma\curvearrowright X$ by homeomorphisms induces a locally compact
Hausdorff \'{e}tale groupoid 
\[
X\rtimes\Gamma\coloneqq\{(\gamma x,\gamma,x):\gamma\in\Gamma,x\in X\}
\]
equipped with the relative topology as a subset of $X\times\Gamma\times X$.
In addition, $(\gamma x,\gamma,x)$ and $(\beta y,\beta,y)$ are composable
only if $\beta y=x$ and 
\[
(\gamma x,\gamma,x)(\beta y,\beta,y)=(\gamma\beta y,\gamma\beta y,y).
\]
One also defines $(\gamma x,\gamma,x)^{-1}=(x,\gamma^{-1},\gamma x)$
and announces that $\GU\coloneqq\{(x,e_{\Gamma},x):x\in X\}$. It
is not hard to verify that $s(\gamma x,\gamma,x)=x$ and $r(\gamma x,\gamma,x)=\gamma x$.
The groupoid $X\rtimes\Gamma$ is called a \textit{transformation
groupoid}. 
\end{example}
The following are several basic properties of locally compact Hausdorff
\'{e}tale groupoids whose proofs could be found in \cite{Sims-groupoids}.
\begin{prop}
Let $\CG$ be a locally compact Hausdorff \'{e}tale groupoid. Then $\GU$
is a clopen set in $\CG$. 
\end{prop}

\begin{prop}
Let $\CG$ be a locally compact Hausdorff \'{e}tale groupoid. Suppose
$U$ and $V$ are open bisections in $\CG$. Then $UV=\{\alpha\beta\in\CG:(\alpha,\beta)\in\CG^{(2)}\cap U\times V\}$
is also an open bisection. 
\end{prop}
It is also convenient to define, for $n=1,2,\ldots$, the set of \emph{composable
$n$-tuples} 
\[
\CG^{(n)}=\left\{ (x_{1},\ldots,x_{n})\in\CG^{n}:s\left(x_{i}\right)=r\left(x_{i+1}\right)\mbox{ for }i=1,2,\ldots,n-1\right\} 
\]
and the \emph{$n$-ary multiplication map}
\[
\delta^{(n)}:\CG^{(n)}\to\CG,\quad(x_{1},\ldots,x_{n})\mapsto x_{1}\cdots x_{n}.
\]

\begin{cor}
\label{cor:multiplication-is-local-homeo}Let $\CG$ be a locally
compact Hausdorff \'{e}tale groupoid. Then for any $n\in\{1,2,\ldots\}$,
the\emph{ }$n$-ary multiplication map is a local homeomorphism. 
\end{cor}
We also record the following useful fact about local homeomorphisms. 
\begin{lem}
\label{lem:local-homeomorphism-cover}Let $f:X\to Y$ be a local homeomorphism
between topological spaces with $Y$ being Hausdorff. Then for any
$y\in Y$ and any compact subset $K\subseteq X$, there are an open
neighborhood $U$ of $y$ in $Y$ and a finite family of open subsets
$V_{1},\ldots,V_{n}$ in $X$ such that 
\begin{enumerate}
\item the map $f$ restricts to a homeomorphism between $V_{i}$ and $U$,
for any $i\in\left\{ 1,\ldots,n\right\} $, and
\item we have $f^{-1}(U)\cap K\subseteq V_{1}\cup\ldots\cup V_{n}$. 
\end{enumerate}
\end{lem}
\begin{proof}
Since $f$ is a local homeomorphism, we know for any $x\in f^{-1}(y)$,
there are open neighborhoods $V_{x}$ of $x$ and $U_{x}$ of $y$
such that $f$ restricts to a homeomorphism between $V_{x}$ and $U_{x}$.
Since the collection $\left\{ f^{-1}\left(Y\setminus\{y\}\right),V_{x}:x\in f^{-1}(y)\right\} $
form an open cover of $K$, by compactness, there are $x_{1},\ldots,x_{n}\in f^{-1}(y)$
such that $\left\{ f^{-1}\left(Y\setminus\{y\}\right),V_{x_{i}}:i=1,\ldots,n\right\} $
form a finite open cover of $K$. Let $L=K\setminus\bigcup_{i=1}^{n}V_{x_{i}}$,
which is a closed subset of $K$ and thus also compact; so is the
image $f\left(L\right)$. Observe that $L\subseteq f^{-1}\left(Y\setminus\{y\}\right)$,
i.e., $y\not\in f(L)$. Since a Hausdorff space has separation between
a point and a compact set, there is an open neighborhood $W$ of $y$
in $Y$ such that $W\cap f(L)=\varnothing$. Let $U=W\cap\left(\bigcap_{i=1}^{n}U_{x_{i}}\right)$
and let $V_{i}=\left(f\mid_{V_{x_{i}}}\right)^{-1}\left(U\right)$,
for $i=1,\ldots,n$. They clearly satisfy the first condition. As
for the second condition, we observe that $f^{-1}(U)\cap L=\varnothing$
and thus $f^{-1}(U)\cap K=f^{-1}(U)\cap\left(L\cup V_{1}\cup\ldots\cup V_{n}\right)\subseteq V_{1}\cup\ldots\cup V_{n}$. 
\end{proof}
For any set $D\subset\GU$, Denote by 
\[
\CG_{D}\coloneqq\{\gamma\in\CG:s(\gamma)\in D\},\ \CG^{D}\coloneqq\{\gamma\in\CG:r(\gamma)\in D\},\ \text{and}\ \ \CG_{D}^{D}\coloneqq\CG^{D}\cap\CG_{D}.
\]
For the singleton case $D=\{u\}$, we write $\CG_{u}$, $\CG^{u}$
and $\CG_{u}^{u}$ instead for simplicity. In this situation, we call
$\CG_{u}$ a \textit{source fiber} and $\CG^{u}$ a \textit{range
fiber}. In addition, each $\CG_{u}^{u}$ is a group, which is called
the \textit{isotropy} at $u$. We also denote by \[\opIso(\CG)=\bigcup_{u\in \GU}\CG^u_u=\{x\in \CG: s(x)=r(x)\}\] the isotropy of the groupoid $\CG$. We say a groupoid $\CG$ is \textit{principal}
if $\opIso(\CG)=\GU$. A groupoid $\CG$ is called \textit{topologically
principal} if the set $\{u\in\GU:\CG_{u}^{u}=\{u\}\}$ is dense in
$\GU$. The groupoid $\CG$ is also said to be \textit{effective} if $\opIso(\CG)^o=\GU$. Recall that effectiveness is equivalent to topological principalness  if $\CG$ is second countable (See \cite[Lemma 4.2.3]{Sims-groupoids}). Therefore, effectiveness is equivalent to the topological freeness of an action of a countable discrete group acting on a compact metrizable space by looking at the corresponding transformation groupoid.

A subset $D$ in $\GU$ is called $\CG$-\textit{invariant}
if $r(\CG D)=D$, which is equivalent to the condition $\CG^{D}=\CG_{D}$.
Note that $\CG|_{D}\coloneqq\CG_{D}^{D}$ is a subgroupoid of $\CG$
with the unit space $D$ if $D$ is a $\CG$-invariant set in $\GU$.
A groupoid $\CG$ is called \textit{minimal} if there are no proper
non-trivial closed $\CG$-invariant subsets in $\GU$.

\begin{notation}
	Throughout the paper, we write $B\sqcup C$ to indicate that
	the union of sets $B$ and $C$ is a disjoint union. In addition,
	we denote by $\bigsqcup_{i\in I}B_{i}$ for the disjoint union of
	the family $\{B_{i}:i\in I\}$. We also denote by $\ceilstar{\cdot}$
	the ceiling function and by $\floorstar{\cdot}$ the floor function
	from $[0,\infty)$ to $\N$. 
	For two collections $\CB$ and $\CC$ of subsets in a space $X$, we write $\CB \vee \CC := \left\{ B \cap C \colon B \in \CB, C \in \CC \right\}$. 
\end{notation}

Finally, we make the following assumption to simplify terminology.

\begin{notation} \label{standing-assumption-paper}
	Throughout the paper, we mean by ``\'{e}tale groupoids'' locally compact, Hausdorff, \'{e}tale topological groupoids. 
\end{notation}

\section{Amenability of extended coarse spaces}\label{sec:metric-amenability}

In this section, we recall and study the amenability of (uniformly
locally finite) extended metric spaces from a coarse geometric point
of view. In particular, we introduce a strengthening of the notion
of metric amenability called \emph{ubiquitous (metric) amenability},
which will be a central tool in our investigation of coarse structures
of groupoids. In particular, we prove a pair of lemmas at the end
of the section that display how ubiquitous amenability and non-amenability
lead to constrasting behaviors on bounded enlargements of arbitrary
finite subsets in metric spaces. 
\begin{defn}
\label{def:metric-boundaries}Recall an \textit{extended} metric space
is a metric space in which the metric is allowed to take the value
$\infty$. An extended metric space admits a unique partition into
ordinary metric spaces, called its \emph{coarse connected components},
such that two points have finite distance if and only if they are
in the same coarse connected component. An extended metric space is
called \textit{locally finite} if any bounded set has finite cardinality. 

Let $(X,d)$ be a locally finite extended metric space and $A$ be
a subset of $X$. For any $R>0$ we define the following boundaries
of $A$:
\begin{enumerate}[label=(\roman*)]
\item \emph{outer $R$-boundary}: $\partial_{R}^{+}A=\{x\in X\setminus A:d(x,A)\leq R\}$;
\item \emph{inner $R$-boundary}: $\partial_{R}^{-}A=\{x\in A:d(x,X\setminus A)\leq R\}$; 
\item \emph{$R$-boundary}: $\partial_{R}A=\{x\in X:d(x,A)\leq R\ \textrm{and}\ d(x,X\setminus A)\leq R\}$. 
\end{enumerate}
\end{defn}
\begin{rem}
\label{3.1} Let $(X,d)$ be an extended metric space. Suppose $A\subset X$
and $R>0$. It is straightforward to see $\partial_{R}^{+}A\subset\bar{B}(\partial_{R}^{-}A,R)$
and $\partial_{R}^{-}A\subset\bar{B}(\partial_{R}^{+}A,R)$. 
\end{rem}
The following concept of amenability of extended metric spaces was
introduced in \cite{BlockWeinberger1992Aperiodic} by Block and Weinberger
and further studied in \cite{AraLiLledoWu2018Amenability} by Ara,
Li, Lled\'{o}, and the second author.
\begin{defn}
\label{3.2} \label{def:metric-amenability}Let $(X,d)$ be a extended
locally finite metric space. 
\begin{enumerate}[label=(\roman*)]
\item For $R>0$ and $\varepsilon>0$, a finite non-empty set $F\subset X$
is called $(R,\varepsilon)$-F{\o}lner if it satisfies 
\[
\frac{|\partial_{R}F|}{|F|}\leq\varepsilon.
\]
We denote by $\opFol(R,\varepsilon)$ the collection of all $(R,\varepsilon)$-F{\o}lner
sets. 
\item The space $(X,d)$ is called \textit{amenable} if, for every $R>0$
and $\varepsilon>0$, there exists a $(R,\varepsilon)$-F{\o}lner set. 
\end{enumerate}
\end{defn}
The following elementary lemma shows that F{\o}lner sets can always be
``localized'' to a single coarse connected component. 
\begin{lem}
\label{lem:Folner-components}Let $(X,d)$ be a extended locally finite
metric space and let $X_{i}$, $i\in I$, be its coarse connected components.
Fix $R,\eps>0$ and let $F$ be an $(R,\varepsilon)$-F{\o}lner set of $X$.
Write $F_{i}=F\cap X_{i}$ for each $i\in I$. Then there is an $i_{0}\in I$
such that $F_{i_{0}}$ is also an $(R,\varepsilon)$-F{\o}lner set. \end{lem}
\begin{proof}
Suppose for any $i\in I$, the set $F_{i}$ is not an $(R,\varepsilon)$-F{\o}lner
set, i.e., either $F_{i}=\varnothing$ or $\left|\partial F_{i}\right|>\varepsilon\left|F_{i}\right|$.
Observe that $\grave{\partial F=\bigsqcup}_{i\in I}\partial F_{i}$
and only finitely many among the $F_{i}$'s are non-empty. Thus we
would have 
\[
\left|\partial F\right|=\sum_{i\in I}\left|\partial F_{i}\right|>\sum_{i\in I}\varepsilon\left|F_{i}\right|=\varepsilon\left|F\right|,
\]
a contradiction to the assumption that $F$ is an $(R,\varepsilon)$-F{\o}lner
set. 
\end{proof}
In this paper, we also need the following stronger version of this
amenability.
\begin{defn}
\label{3.0}\label{def:uniform-metric-amenability} An extended metric
space $(X,d)$ is called \textit{ubiquitously amenable} (or \emph{ubiquitously
metrically amenable}) if, for every $R>0$ and $\varepsilon>0$, there
exists an $S>0$ such that for any $x\in X$, there is an $(R,\varepsilon)$-F{\o}lner
set $F$ in the ball $\bar{B}(x,S)$. 
\end{defn}
An extended metric space $(X,d)$ is called \textit{uniformly locally
finite}\footnote{This notion also appeared as \textit{bounded geometry} in the literature.}
if for any $R>0$, there is a uniform finite upper bound on the cardinalities
of all closed balls with radius $R$. i.e., 
\[
\sup_{x\in X}|\bar{B}(x,R)|<\infty.
\]
To simplify the notation, for uniformly locally finite space $(X,d)$,
we define a function $\FN_{(X,d)}:\mathbb{R}^{+}\to\mathbb{N}$ by 
\begin{equation}
\FN_{(X,d)}(r) :=\sup_{x\in X}|\bar{B}(x,r)|.\label{eq:N(r)}
\end{equation}
We may write $\FN_{X}$ or simply $\FN$ for $\FN_{(X,d)}$ if there is no risk of confusion. 
The following criteria for establishing amenability for uniformly
locally finite extended metric spaces is straightforward but useful. 
\begin{prop}
\label{3.3} \label{prop:metric-amenable-change-boundaries}
Let $(X,d)$ be a uniformly locally finite extended metric
space. 
The following are equivalent: 
\begin{enumerate}[label=(A\arabic*)]
\item \label{prop:metric-amenable-change-boundaries:A-original} The space $(X,d)$ is amenable. 
\item \label{prop:metric-amenable-change-boundaries:A-outer} For any $R,\varepsilon>0$, there is a finite set $F\subset X$ such that
$|\del_{R}^{+}F|\leq\varepsilon|F|$. 
\item \label{prop:metric-amenable-change-boundaries:A-inner} For any $R,\varepsilon>0$, there is a finite set $F\subset X$ such that
$|\del_{R}^{-}F|\leq\varepsilon|F|$. 
\item \label{prop:metric-amenable-change-boundaries:A-nbhd} For any $R,\varepsilon>0$, there is a finite set $F\subset X$ such that
$|\bar{B}(F,R)|\leq(1+\varepsilon)|F|$. 
\end{enumerate}
In addition, the following are equivalent: 
\begin{enumerate}[label=(U\arabic*)]
\item \label{prop:metric-amenable-change-boundaries:U-original} The space $(X,d)$ is ubiquitously amenable. 
\item \label{prop:metric-amenable-change-boundaries:U-outer} For any $R,\varepsilon>0$, there exists an $S>0$ such that for any $x\in X$, there is a finite set $F\subset X$ in the ball $\bar{B}(x,S)$ such that
$|\del_{R}^{+}F|\leq\varepsilon|F|$. 
\item \label{prop:metric-amenable-change-boundaries:U-inner} For any $R,\varepsilon>0$, there exists an $S>0$ such that for any $x\in X$, there is a finite set $F\subset X$ in the ball $\bar{B}(x,S)$ such that
$|\del_{R}^{-}F|\leq\varepsilon|F|$. 
\item \label{prop:metric-amenable-change-boundaries:U-nbhd} For any $R,\varepsilon>0$, there exists an $S>0$ such that for any $x\in X$, there is a finite set $F\subset X$ in the ball $\bar{B}(x,S)$ such that
$|\bar{B}(F,R)|\leq(1+\varepsilon)|F|$. 
\end{enumerate}
\end{prop}
\begin{proof}
The proof is straightforward by observing that for any $R>0$ and
$F\subset X$ one has $\del_{R}F=\del_{R}^{+}F\cup\del_{R}^{-}F$
and $\del_{R}^{+}F=\bar{B}(F,R)\setminus F$ as well as Remark \ref{3.1} and the ensuing facts that
$|\del_{R}^{+}F|\leq\FN_{X}(R)|\del_{R}^{-}F|$ and $|\del_{R}^{-}F|\leq\FN_{X}(R)|\del_{R}^{+}F|$. 
\end{proof}
The following lemma is useful in establishing Proposition \ref{3.5}
and \ref{3.7}.
\begin{lem}
\label{3.4} Suppose that $(X,d)$ is a uniformly locally finite extended
metric space. Let $s>0$ and $n\in\mathbb{N}$. Then for any finite
set $F\subset X$ satisfying $|F|\geq n\cdot\FN_{X}(s)$, there exist
distinct points $x_{1},\dots,x_{n}\in F$ such that, for all $i,j\in\{1,\dots,n\}$
with $i\neq j$, one has $d(x_{i},x_{j})>s$. In particular there
exists at least $n$ many disjoint balls $\bar{B}(x_{i},s/2)$ for
$i=1,\dots,n$ in $\bar{B}(F,s/2)$. \end{lem}
\begin{proof}
We choose points $x_{1},\dots,x_{n}$ by induction. First pick $x_{1}\in F$.
Suppose that $x_{1},\dots,x_{k}$ has been defined such that $d(x_{i},x_{j})>s$
for all $1\leq i\neq j\leq k$. Then observe that 
\[
|\bigcup_{i=1}^{k}\bar{B}(x_{i},s)|\leq\sum_{i=1}^{k}|\bar{B}(x_{i},s)|\leq k\cdot\FN_{X}(s).
\]
Then choose $x_{k+1}\in F\setminus\bigcup_{i=1}^{k}\bar{B}(x_{i},s)$
whose cardinality satisfies 
\[
|F\setminus\bigcup_{i=1}^{k}\bar{B}(x_{i},s)|\geq(n-k)\cdot\FN_{X}(s).
\]
This finishes the proof. 
\end{proof}

The following proposition provides a major upgrade to the potency of ubiquitous amenability: instead of merely being able to find F\o{}lner sets (within a uniform distance) around all points, it allows us to turn all sets into F\o{}lner sets by tweaking their boundaries within a uniform distance. The groupoid incarnation of this result, Theorem~\ref{thm:fiberwise-amenable-dichotomy}\eqref{thm:fiberwise-amenable-dichotomy:amenable}, will be a crucial tool in the sequel paper \cite{MWAE}. 

\begin{prop}
\label{3.5} Let $(X,d)$ be a ubiquitously amenable and uniformly
locally finite extended metric space. Then for any $r \geq 0$ and any $\varepsilon>0$, 
there exists $S \geq 0$ such that for any finite subset $M\subseteq X$
there exists a finite set $F$ with $M\subseteq F\subseteq\bar{B}(M,S)$
and 
\[
\frac{|\partial_{r}^{+}F|}{|F|}\leq\varepsilon \quad \text{ or, equivalently, } \quad \frac{|\bar{B}(F,r)|}{|F|}\leq 1+\varepsilon.
\]
\end{prop}
\begin{proof}
Let $r \geq 0$ and $\varepsilon>0$ be given. 
Since $(X,d)$ is ubiquitously amenable,
there is an $s \geq 0$ such that for all $x\in X$ there is a finite set
$F_{x}\subseteq\bar{B}(x,s)$ such that 
\[
\frac{|\partial_{r}^{+}F_{x}|}{|F_{x}|}\leq\varepsilon/2.
\]
Now define 
\[
n=\ceilstar{\frac{2\FN_{X}(r)}{\varepsilon}}\cdot\FN_{X}(2s),\ l = \ceilstar{\log_{(1+\varepsilon)}n},\ 
\textrm{and}\ \ S=l r+s+1.
\]
To show this choice of $S$ satisfies the requirement, we fix a finite set $M$ in $X$ and aim to find $F$ that satisfies the desired condition. 

To this end, we define $F_{0}=M$ and, recursively, $F_{k+1}=\bar{B}(F_{k},r)$ for $k = 1, 2, \ldots$.  
It is easily proved by induction that $M \subseteq F_{k} \subseteq \bar{B}(M,k r)$ for any $k \in \N$. 
We discuss two cases: 

\begin{enumerate}
	\item[Case 1:] If there is a $k_0 \in \left\{ 0, 1, \ldots, l - 1 \right\}$ such that 
	\[
		\frac{|\partial_{r}^{+}F_{k_0}|}{|F_{k_0}|}\leq\varepsilon \; ,
	\]
	then it suffices to choose $F := F_{k_0}$, since 
	\[
		M \subseteq F_{k_0} \subseteq \bar{B}(M,k_0 r) \subseteq \bar{B}(M,S) \; .
	\]
	\item[Case 2:] Now suppose, on the contrary, for any $k \in \left\{ 0, 1, \ldots, l - 1 \right\}$, we have 
	\[
		\frac{|\partial_{r}^{+}F_k|}{|F_k|}>\varepsilon
	\]
	and thus in particular 
	\[
		\frac{|F_{k+1}|}{|F_{k}|} = \frac{|\bar{B}(F_k,r)|}{|F_k|}>1+\varepsilon \; .
	\]
	It follows that 
	\begin{align*}
		\frac{|\bar{B}(M,l r)|}{|M|} \geq \frac{|F_{m}|}{|F_{0}|} = \prod_{k=0}^{m-1} \frac{|F_{k+1}|}{|F_{k}|} \geq (1+\varepsilon)^{\ceilstar{\log_{(1+\varepsilon)}n}}>n=\ceilstar{\frac{2\FN_{X}(r)}{\varepsilon}}\cdot\FN_{X}(2s) \; .
	\end{align*}
	Now write $m=\ceilstar{2\FN_{X}(r)/\varepsilon}\cdot|M|$. Lemma \ref{3.4}
	implies that there are distinct points $x_{1},\dots,x_{m}\in\bar{B}(M,l r)$
	such that $d(x_{i},x_{j})>2s$. 
	Noting that $F_{x_{i}}\subset\bar{B}(x_{i},s)$ for $i = 1, \ldots, m$, we may write 
	\[
		F := M\cup(\bigsqcup_{i=1}^{m}F_{x_{i}})\subseteq\bar{B}(M,S)
	\]
	Then we have 
	\begin{align*}
	\frac{|\partial_{r}^{+}(F)|}{|F|} & \leq\frac{|\partial_{r}^{+}(M)|+\sum_{i=1}^{m}|\partial_{r}^{+}(F_{x_{i}})|}{\sum_{i=1}^{m}|F_{x_{i}}|}\\
	& =\frac{|\partial_{r}^{+}(M)|}{|M|}\cdot\frac{|M|}{\sum_{i=1}^{m}|F_{x_{i}}|}+\sum_{i=1}^{m}\frac{|\partial_{r}^{+}(F_{x_{i}})|}{|F_{x_{i}}|}\cdot\frac{|F_{x_{i}}|}{\sum_{i=1}^{m}|F_{x_{i}}|}.
	\end{align*}
	Since $|\partial_{r}^{+}(M)|\leq|\bar{B}(M,r)|\leq\FN_{X}(r)|M|$ and
	\[
	\frac{|M|}{\sum_{i=1}^{m}|F_{x_{i}}|}\leq\frac{|M|}{m}\leq\frac{\varepsilon}{2\FN_{X}(r)},
	\]
	one has 
	\[
	\frac{|\partial_{r}^{+}(M)|}{|M|}\cdot\frac{|M|}{\sum_{i=1}^{m}|F_{x_{i}}|}\leq\frac{\varepsilon}{2}.
	\]
	On the other hand, by the definition of all $F_{x_{i}}$ one has 
	\[
	\sum_{i=1}^{m}\frac{|\partial_{r}^{+}(F_{x_{i}})|}{|F_{x_{i}}|}\cdot\frac{|F_{x_{i}}|}{\sum_{i=1}^{m}|F_{x_{i}}|}\leq\frac{\varepsilon}{2}\cdot\sum_{i=1}^{m}\frac{|F_{x_{i}}|}{\sum_{i=1}^{m}|F_{x_{i}}|}=\frac{\varepsilon}{2}.
	\]
	This implies that 
	\[
	\frac{|\partial_{r}^{+}(F)|}{|F|}\leq\varepsilon,
	\]
\end{enumerate}
This completes the proof. 
\end{proof}
We find Case~2 of the above proof surprising in that here we constructed a F\o{}lner set out of a Ponzi-scheme-like structure, which is usually considered antithetical to being F\o{}lner. 

In contrast to Proposition~\ref{3.5}, we show below a paradoxical phenomenon for sets in nonamenable
extended metric spaces. 
\begin{prop}
\label{3.7} Let $(X,d)$ be a uniformly locally finite extended metric
space, which is not amenable. For any $r \geq 0$ and any $n\in\mathbb{N}$,
there exists $S \geq 0$ such that for any finite subset $M \subseteq X$, there
are at least $n|M|$ many disjoint $r$-balls, i.e., $\bar{B}(x_{i},r)$ for $i=1,\dots,n|M|$,
contained in $\bar{B}(M,S)$. 
\end{prop}
\begin{proof}
Let $r \geq 0$ and $n\in\mathbb{N}$ be given. Since $(X,d)$ is not amenable,
there is an $\varepsilon>0$ and $R  > 0$ such that for any finite subset
$F \subseteq X$ one has 
\[
	\frac{|\bar{B}(F,R)|}{|F|}>1+\varepsilon.
\]
Choose a $k\in\mathbb{N}$ such that $(1+\varepsilon)^{k}\geq n\cdot\FN_{X}(2r)$.
Then define $F_{0}:=M$ and, recursively, $F_{k+1}:=\bar{B}(F_{k},R)$ for $k\in\mathbb{N}$. This process implies that 
for any finite subset $M\subseteq X$, one has 
\[
\frac{|\bar{B}(M,kR)|}{|M|}\geq(1+\varepsilon)^{k}\geq n\cdot\FN_{X}(2r).
\]
We write $S=kR+r$. Then Lemma \ref{3.4} implies that there exist
distinct points $x_{1},\dots,x_{n|M|}\in\bar{B}(M,kR)$ such that
$d(x_{i},x_{j})>r$ for any distinct $i, j\in\{1,\dots,n|M|\}$. In particular
there exists at least $n|M|$ disjoint balls $\bar{B}(x_{i},r)$, for
$i=1,\dots,n|M|$, in $\bar{B}(M,S)$. 
\end{proof}

\section{Coarse geometry of \'{e}tale groupoids \label{sec:coarse-geometry}}

A fundamental and motivating fact in coarse geometry is that one can
always assign a length function to a countable discrete group, which
induces a (right-)invariant proper metric on the group, in a way unique
up to coarse equivalence \LyXbar \LyXbar{} for a finitely generate
group, this amounts to taking the graph metric of the Cayley graph
(after fixing a set of generators). In this procedure, the amenability
of the group itself is equivalent to the metric amenability mentioned
in the last section of the resulting metric space. 

Motivated by this, one may establish a similar framework for $\sigma$-compact \'{e}tale groupoids, realizing them
as extended metric spaces by equipping metrics to all the source (or
range) fibers in a uniform and invariant manner. We remark that  fiberwise defined Caylay graph has also been considered in \cite{Nekrashevych2019Simple} for ample groupoids to study topological full groups. We start our discussion
without the topological structure. 
\begin{defn}
\label{def:invariant-fiberwise-extended-metric}An extended metric $\rho$
on a groupoid $\CG$ is \end{defn}
\begin{itemize}
\item \emph{invariant} (or, more precisely, \emph{right-invariant}) if,
for any $x,y,z\in\CG$ with $s(x)=s(y)=r(z)$, we have $\rho(x,y)=\rho(xz,yz)$; 
\item \emph{fiberwise} (or, more precisely, \emph{source-fiberwise}) if,
for any $x,y\in\CG$, we have $\rho(x,y)=\infty$ if and only if $s(x)\neq s(y)$.
\end{itemize}

The significance of these properties is demonstrated in the following characterization. 

\begin{lem}\label{lem:invariant-fiberwise-extended-metric}
	An extended metric $\rho$ on a groupoid $\CG$ is right-invariant and source-fiberwise if and only if 
	\[
		\CG = \bigcup_{r \geq 0} \bar{B}\left(\GU,r\right)
	\]
	and 
	\[
		\bar{B}(M,r) = \bar{B}\left(\GU,r\right) \cdot M \quad \text{ for any } M \subseteq \CG \text{ and } r \geq 0 \; .
	\]
\end{lem}

\begin{proof}
	For the ``only if'' direction, we see, under the assumption that $\rho$ is right-invariant and source-fiberwise, that given any $M \subseteq \CG$, any $r \geq 0$ and any $x \in \CG$, we have: 
	\begin{itemize}
		\item $\rho(x , s(x)) < \infty$, which proves the first desired equality; 
		\item $x \in \bar{B}(M,r)$ if and only if $\rho(x,y) \leq r$ for some $y \in M$ if and only if there exists $y \in M$ such that $s(x) = s(y)$ and $\rho \left( x y^{-1},y y^{-1} \right) \leq r$ if and only if $x = z y$ for some $y \in M$ and some $z \in \CG$ with $\rho(z, s(z)) \leq r$ if and only if $x = z y$ for some $y \in M$ and some $z \in \bar{B}\left(\GU,r\right)$, which proves the second desired equality. 
	\end{itemize}
	
	For the ``if'' direction, we see, under the assumption that the second equality in the statement hold, that for any $x,y\in\CG$ and $r \geq 0$, we have: 
	\begin{itemize}
		\item $\rho(x,y) \leq r$ if and only if $x \in \bar{B}(\{y\},r)$ if and only if $x \in \bar{B}\left(\GU,r\right) \cdot \{y\}$ if and only if $s(x)=s(y)$ and $x y^{-1} \in \bar{B}\left(\GU,r\right)$. 
	\end{itemize}
	It is straightforward to verify that $\rho$ is right-invariant and source-fiberwise using this property above as well as the assumption $\CG = \bigcup_{r \geq 0} \bar{B}\left(\GU,r\right)$. 
\end{proof}

Just as in the case of groups, it is more efficient to encode invariant
metrics by length functions. To the best knowledge of the authors,
the discussion of length functions on \'{e}tale groupoids first appeared
in \cite[Definition 2.21]{Oyono-OyonoYu2019Quantitative}, with ideas
from J.-L.~Tu. Our terminology differs slightly. 
\begin{defn}
\label{def:length-function}Recall a length function on a groupoid
$\CG$ is a function $\ell:\CG\rightarrow[0,\infty)$ satisfying,
for any $x,y\in\CG$, 
\begin{enumerate}[label=(\roman*)]
\item  $\ell(x)=0$ if and only if $x\in\CG^{(0)}$, 
\item (symmetricity) $\ell(x)=\ell(x^{-1})$, and 
\item (subadditivity) $\ell(xy)\leq\ell(x)+\ell(y)$ if $x$ and $y$ are
composable in $\CG$. 
\end{enumerate}
\end{defn}

On a groupoid $\CG$, there is a canonical one-to-one correspondence
between length functions and invariant fiberwise extended metrics.
On the one hand, given any length function $\ell$ on $\CG$, we associate
an extended metric $\rho_{\ell}$ by declaring, for $x,y\in\CG$,
\[
\rho_{\ell}(x,y)=\begin{cases}
\ell(xy^{-1}), & s(x)=s(y)\\
\infty, & s(x)\not=s(y)
\end{cases}.
\]
On the other hand, given any invariant fiberwise extended metric $\rho$
on $\CG$, we associate a function 
\[
\ell_{\rho}:\CG\to[0,\infty),\quad g\mapsto\rho(g,s(g)),
\]
which does not take the value $\infty$ since $\rho$ is fiberwise. 
\begin{lem}
\label{lem:length-function-metric}On a groupoid $\CG$, the above
assignments give rise to a pair of bijections between length functions
and invariant fiberwise extended metrics. \end{lem}
\begin{proof}
It is routine to verify that $\rho_{\ell}$ as defined above is indeed
an extended metric, where positive definiteness and symmetricity of
$\ell$ lead to those of $\rho_{\ell}$ and subadditivity leads to
the triangle inequality. It is also clear that $\rho_{\ell}$ is invariant
and fiberwise. On the other hand, to verify that $\ell_{\rho}$ is
a length function, we see, with the help of invariance, the same correspondence
between the conditions in the opposite direction. 
\end{proof}
Now we focus on \'{e}tale groupoids. We show, in analogy with the case
of groups, that a $\sigma$-compact \'{e}tale
groupoid determines, up to coarse equivalence, a canonical invariant
fiberwise extended metric that enjoys the following properties. 
\begin{defn}
\label{def:coarse-length-function}Let $\ell:\CG\to[0,\infty)$ be
a length function on an \'{e}tale groupoid $\CG$. For any subset $K\subseteq\CG$,
we write 
\[
\overline{\ell}(K)=\sup_{x\in K}\ell(x).
\]
We say $\ell$ is 
\begin{itemize}
\item \textit{proper} if, for any $K\subset\CG\setminus\CG^{(0)}$, $\overline{\ell}(K)<\infty$
implies that $K$ is precompact, 
\item \textit{controlled} if, for any $K\subset\CG$, $\overline{\ell}(K)<\infty$
is implied by that $K$ is precompact, and
\item \emph{coarse} if it is both proper and controlled. 
\item \emph{continuous }if it is a continuous function with regard to the
topology of $\CG$. 
\end{itemize}
Two length functions $\ell_{1},\ell_{2}$ are said to be \emph{coarsely
equivalent} if for any $r>0$, we have 
\[
\sup\left\{ \ell_{1}(x):\ell_{2}(x)\leq r\right\} <\infty\quad\mbox{and}\quad\sup\left\{ \ell_{2}(x):\ell_{1}(x)\leq r\right\} <\infty.
\]

\end{defn}
It is straightforward to see that coarse equivalence of length functions
is indeed an equivalence relation, and a continuous length function
is controlled. We may also express coarse equivalence using \emph{control
functions}, as is common in coarse geometry. 
\begin{lem}
\label{lem:coarse-equivalence-length}Two length functions $\ell_{1},\ell_{2}$
are coarsely equivalent if and only if there are non-decreasing unbounded
functions $f_{+},f_{-}:[0,\infty)\to[0,\infty)$ (sometimes referred
to as \emph{control functions}) such that 
\begin{equation}
f_{-}\left(\ell_{1}(x)\right)\leq\ell_{2}(x)\leq f_{+}\left(\ell_{1}(x)\right)\quad\mbox{for any }x\in\CG.\label{eq:coarse-equivalence-length-control}
\end{equation}
Moreover, we may also assume $f_{+}(0)=f_{-}(0)=0$ in the above. \end{lem}
\begin{proof}
Assuming there are non-decreasing unbounded functions $f_{+},f_{-}:[0,\infty)\to[0,\infty)$
satisfying \eqref{eq:coarse-equivalence-length-control}, then for
any $r>0$, we have
\[
\sup\left\{ \ell_{2}(x):\ell_{1}(x)\leq r\right\} \leq\sup\left\{ f_{+}\left(\ell_{1}(x)\right):\ell_{1}(x)\leq r\right\} \leq f_{+}(r)<\infty,
\]
and 
\[
\sup\left\{ \ell_{1}(x):\ell_{2}(x)\leq r\right\} \leq\sup\left\{ \ell_{1}(x):f_{-}\left(\ell_{1}(x)\right)\leq r\right\} \leq\sup\left\{ s:f_{-}\left(s\right)\leq r\right\} <\infty,
\]
thanks to the unboundedness of $f_{-}$. 

On the other hand, assuming $\ell_{1},\ell_{2}$ are coarsely equivalent
as above, we may define 
\[
f_{+}(r)=\sup\left\{ r,\ell_{2}(x):\ell_{1}(x)\leq r\right\} \quad\mbox{and}\quad f_{-}(r)=\inf\left\{ r,\ell_{2}(x):\ell_{1}(x)\geq r\right\} 
\]
for $r\in[0,\infty)$. It is immediate that both functions are nondecreasing,
$f_{+}$ is unbounded, $f_{+}(0)=f_{-}(0)=0$, and \eqref{eq:coarse-equivalence-length-control}
is satisfied. To see $f_{-}$ is unbounded, we observe that for any
$r,s\geq0$, we have $f_{-}(r)<s$ if and only if either $r<s$ or
there is $x\in\CG$ such that $\ell_{2}(x)<s$ but $\ell_{1}(x)\geq r$,
the latter possibility implying $r\leq\sup\left\{ \ell_1(x):\ell_{2}(x)<s\right\} $.
This shows that $f_{-}^{-1}([0,s])$ is bounded for any $s\geq0$,
i.e., $f_{-}$ is unbounded. 
\end{proof}

\begin{rem}
\label{rem:coarse-equivalence-metric}Under the correspondence of
Lemma~\ref{lem:length-function-metric}, coarse equivalence of two
length functions $\ell_{1}\mbox{ and }\ell_{2}$ translates to coarse
equivalence of their induced extended metrics $\rho_{\ell_{1}}$ and
$\rho_{\ell_{2}}$, that is, we have 
\[
\sup\left\{ \rho{}_{\ell_{1}}(x,y):\rho_{\ell_{2}}(x,y)\leq r\right\} <\infty\quad\mbox{and}\quad\sup\left\{ \rho{}_{\ell_{2}}(x,y):\rho_{\ell_{1}}(x,y)\leq r\right\} <\infty,
\]
or equivalently, there are nondecreasing unbounded functions $f_{+},f_{-}:[0,\infty)\to[0,\infty)$
(sometimes referred to as \emph{control functions}) such that 
\[
f_{-}\left(\rho{}_{\ell_{1}}(x,y)\right)\leq\rho_{\ell_{2}}(x,y)\leq f_{+}\left(\rho{}_{\ell_{1}}(x,y)\right)\quad\mbox{for any }x,y\in\CG,
\]
where we adopt the convention that $f_{+}(\infty)=f_{-}(\infty)=\infty$. 
\end{rem}

\begin{lem}
\label{lem:coarse-length-unique} Any two coarse length functions
on an \'{e}tale groupoid $\CG$ are coarsely equivalent to each other. \end{lem}
\begin{proof}
Given two coarse length functions $\ell_{1},\ell_{2}$ on $\CG$,
we see that for any $r>0$, by the properness of $\ell_{1}$, the
set $\left\{ g\in\CG\setminus\CG^{(0)}:\ell_{1}(g)\leq r\right\} $
is precompact, and thus by the facts that $\ell_{2}$ is controlled
and $\ell_{2}(\CG^{(0)})=\{0\}$, we have $\sup\left\{ \ell_{2}(g):\ell_{1}(g)\leq r\right\} <\infty$.
Similarly, we have $\sup\left\{ \ell_{1}(g):\ell_{2}(g)\leq r\right\} <\infty$,
as desired. 
\end{proof}

\begin{rem}
	\label{rem:coarse-length-bounded} A coarse length function on an \'{e}tale groupoid $\CG$ is bounded if and only if $\CG \setminus \CG^{(0)}$ is compact. This follows directly from Definition~\ref{def:coarse-length-function} and the fact that $\CG^{(0)}$ is open in $\CG$. 
\end{rem}

To prove the existence of coarse continuous length functions, we make
use of the following simple topological fact. We include the proof
for completeness. Recall a continuous function $g$ between two topological spaces is called \textit{proper} if $g^{-1}(K)$ is compact for any compact set $K$.
\begin{lem}
\label{lem:proper-function}Let $X$ be a $\sigma$-compact, locally
compact and Hausdorff space. Then there exists a continuous proper
map $g:X\to[0,\infty)$. \end{lem}
\begin{proof}
Choose a sequence of compact subsets 
\[
K_{0}\subset K_{1}\subset\ldots\subset X
\]
with $K_{i}\subseteq\interior{K_{i+1}}$ for each $i$ and $X=\bigcup_{i=0}^{\infty}K_{i}$.
It follows that the closed sets $K_{0},\partial K_{1},\partial K_{2},\ldots$
are disjoint, allowing us to define $g(K_{0})=\{0\}$ and $g\left(\partial K_{i}\right)=\{i\}$
for $i=1,2,\ldots$. Applying the Tietze extension theorem to the
compact Hausdorff spaces $K_{i+1}\setminus\interior{K_{i}}$, for
$i=0,1,2,\ldots$, we obtain a continuous function $g:X\to[0,\infty)$
mapping $K_{i+1}\setminus\interior{K_{i}}$ into $[i,i+1]$, for $i=0,1,2,\ldots$,
which implies it is proper. \end{proof}
We remark that such a continuous proper function $g$ is bounded if and only if $X$ is bounded. 
\begin{thm}
\label{thm:coarse-length-functions}Up to coarse equivalence, any
$\sigma$-compact \'{e}tale groupoid has a unique
coarse continuous length function. \end{thm}
\begin{proof}
Uniqueness up to coarse equivalence follows from Lemma~\ref{lem:coarse-length-unique}.
It remains to show existence. To get started, we choose a continuous
function $f:\CG\to\{0\}\cup[1,\infty)$ such that $f^{-1}(\{0\})=\CG^{(0)}$,
$f(x)=f\left(x^{-1}\right)$ for any $x\in\CG$, and $f\mid_{\CG\setminus\CG^{(0)}}$
is proper, i.e., for any $r\geq1$, the inverse image $f^{-1}\left([1,r]\right)$
is compact. Indeed, to construct $f$, we first observe that since
$\CG^{(0)}$ is a clopen subset of $\CG$, the complement $\CG\setminus\CG^{(0)}$
is also $\sigma$-compact, locally compact and Hausdorff, which enables
us to apply Lemma~\ref{lem:proper-function} to obtain a proper continuous
function $g:\CG\setminus\CG^{(0)}\to[0,\infty)$ and then define 
\[
f(x)=\begin{cases}
0, & x\in\CG^{(0)}\\
1+\frac{g(x)+g\left(x^{-1}\right)}{2}, & x\in\CG\setminus\CG^{(0)}
\end{cases},
\]
which clearly satisfies all the requirements. 

We define a function $\ell:\CG\rightarrow[0,\infty)$ by 
\[
\ell(x)=\inf\left\{ \sum_{j=1}^{k}f\left(y_{j}\right):k\in\N\mbox{ and }y_{1},\ldots,y_{k}\in\CG\setminus\CG^{(0)}\mbox{ such that }x=y_{1}\dots y_{k}s(x)\right\} 
\]
for $x\in\CG$, where the degenerate case of $k=0$ corresponds to
$x=s(x)$ and $\ell(x)=0$. It is immediate that the function $\ell$
defined above is a length function on $\CG$. 

To study $\ell$, we describe an equivalent definition of it in terms
of the sets $\CG^{(n)}$ of composable $n$-tuples and the $n$-ary
multiplication maps $\delta^{(n)}$. For $n=1,2,\ldots$, define 
\[
\mathring{\CG}^{(n)}=\CG^{(n)}\cap\left(\CG\setminus\CG^{(0)}\right)^{n}=\left\{ (x_{1},\ldots,x_{n})\in\CG^{(n)}:x_{1},\ldots,x_{n}\in\CG\setminus\CG^{(0)}\right\} ,
\]
\[
\mathring{\delta}^{(n)}=\delta^{(n)}\mid_{\mathring{\CG}^{(n)}}:\mathring{\CG}^{(n)}\to\CG,\quad(x_{1},\ldots,x_{n})\mapsto x_{1}\cdots x_{n},
\]
\[
f^{(n)}:\CG^{(n)}\to[0,\infty),\quad(x_{1},\ldots,x_{n})\mapsto\sum_{j=1}^{k}f\left(x_{j}\right),
\]
\[
\mathring{f}^{(n)}=f^{(n)}\mid_{\mathring{\CG}^{(n)}}:\mathring{\CG}^{(n)}\to[0,\infty).
\]
Combining these definitions, we have
\[
\ell(x)=\inf\left(\left\{ f(x)\right\} \cup\bigcup_{j=2}^{\infty}\mathring{f}^{(j)}\left(\left(\mathring{\delta}^{(j)}\right)^{-1}(\{x\})\right)\right)\quad\mbox{for any }x\in\CG.
\]
Moreover, observe that for $n=1,2,\ldots$, the range of $\mathring{f}^{(n)}$
is contained in $[n,\infty)$ and, for any $r\geq1$, we have 
\[
\left(\mathring{f}^{(n)}\right)^{-1}([1,r])\subseteq\left(f^{-1}([1,r])\right)^{n}.
\]
Hence for any $x\in\CG$ and $N\in\N$ satisfying $\ell(x)\leq N$,
we may remove large values that do not affect the infimum and obtain
\begin{equation}
\ell(x)=\inf\left(\left\{ f(x)\right\} \cup\bigcup_{j=2}^{N}\mathring{f}^{(j)}\left(\left(f^{-1}([1,N])\right)^{j}\cap\left(\mathring{\delta}^{(j)}\right)^{-1}(\{x\})\right)\right).\label{eq:length-formula-bounded}
\end{equation}

To show that $\ell$ is proper, let $K\subset\CG\setminus\CG^{(0)}$
and suppose $\overline{\ell}(K)<N$ for some $N\in\N$. Then for any
$x\in K$, since $f(x)\geq1$, it follows from \eqref{eq:length-formula-bounded}
that there is $n\in\{1,\ldots,N\}$ such that the set 
\[
\left(f^{-1}([1,N])\right)^{n}\cap\left(\mathring{\delta}^{(n)}\right)^{-1}(\{x\})
\]
is non-empty. Therefore we have
\[
K\subseteq\bigcup_{j=1}^{N}\underbrace{f^{-1}([1,N])\cdots f^{-1}([1,N])}_{j},
\]
the latter set being a finite union of products of compact sets, and
thus compact. This shows that $\ell$ defined above is proper.

To show that $\ell$ is continuous, it suffices to prove that for
any $x\in\CG$, there is an open neighborhood $U$ of $x$ such that
$\ell$ is continuous when restricted to $U$. To this end, we let
$N=\ceilstar{\ell(x)+1}$. Thus for any $y$ in the open neighborhood
$f^{-1}([0,N))$ of $x$, the formula \eqref{eq:length-formula-bounded}
applies with $y$ in place of $x$. Observe that for $j=1,2,\ldots N$,
by Corollary~\ref{cor:multiplication-is-local-homeo} and the fact
that $\mathring{\CG}^{(j)}$ is a clopen subset of $\CG^{(j)}$, we
know $\mathring{\delta}^{(j)}$ is a local homeomorphism. Applying
Lemma~\ref{lem:local-homeomorphism-cover} with $\mathring{\delta}^{(j)}$,
$x$, and $\left(f^{-1}([1,N])\right)^{j}$ in place of $f$, $y$,
and $K$, we may find an open neighborhood $U^{(j)}$ of $x$ inside
$f^{-1}([0,N))$ and a finite family of open subsets $V_{1}^{(j)},\ldots,V_{m_{j}}^{(j)}$
in $\mathring{\CG}^{(j)}$ such that $\mathring{\delta}^{(j)}$ restricts
to a homeomorphism between $V_{i}^{(j)}$ and $U^{(j)}$, for any
$i\in\left\{ 1,\ldots,m_{j}\right\} $, and we have 
\[
\left(f^{-1}([1,N])\right)^{j}\cap\left(\mathring{\delta}^{(j)}\right)^{-1}\left(U^{(j)}\right)\subseteq V_{1}^{(j)}\cup\ldots\cup V_{m_{j}}^{(j)}.
\]
Writing $\eta_{i}^{(j)}:U^{(j)}\to V_{i}^{(j)}$ for the inverse of
$\mathring{\delta}^{(j)}\mid_{V_{i}^{(j)}}$, for $i=1,\ldots,m_{j}$.
Then for any $y\in U^{(j)}$, we have 
\[
\left(f^{-1}([1,N])\right)^{j}\cap\left(\mathring{\delta}^{(j)}\right)^{-1}\left(\{y\}\right)=\left(f^{-1}([1,N])\right)^{j}\cap\left\{ \eta_{i}^{(j)}(y):i=1,\ldots,m_{j}\right\} .
\]
Let $U=U_{1}\cap\ldots\cap U_{N}$. Then, for any $y\in U$, we may
rewrite \eqref{eq:length-formula-bounded} as
\begin{equation}
\ell(y)=\min\left\{ f(y),\left(\mathring{f}^{(j)}\circ\eta_{i}^{(j)}\right)(y):j=1,\ldots,N,\ i=1,\ldots,m_{j}\right\} .\label{eq:length-formula-finite-min}
\end{equation}
Hence on the open neighborhood $U$ of $x$, $\ell$ is equal to the
minimum of finite number of continuous functions, and is thus itself
continuous, as desired. 

The controlledness of $\ell$ follows directly from its continuity.
Therefore $\ell$ is a coarse continuous length function. \end{proof}
\begin{rem}\label{4.10}
\label{rem:coarse-length-functions-ample}If a $\sigma$-compact \'{e}tale groupoid $\CG$ is also ample, then we can
choose a coarse length function $\ell$ that is locally constant.
Indeed, when carrying out the proof of Theorem~\ref{thm:coarse-length-functions},
we observe that we can choose the function $g$ and thus also the
function $f$ to be locally constant, by choosing an increasing sequence
of \emph{open }compact subsets $K_{0}\subset K_{1}\subset\ldots\subset\CG\setminus\CG^{(0)}$
with $\CG\setminus\CG^{(0)}=\bigcup_{i=0}^{\infty}K_{i}$ and then
defining $g(x)=\min\left\{ i:x\in K_{i}\right\} $ for all $x\in\CG\setminus\CG^{(0)}$.
It then follows from \eqref{eq:length-formula-finite-min} that $\ell$
is locally equal to the minimum of a finite collection of locally
finite functions, and thus is itself locally constant. 
\end{rem}

\begin{defn}
\label{def:coarse-metric} The unique-up-to-coarse-equivalence
invariant fiberwise extended metric induced by a coarse continuous
length function on a $\sigma$-compact \'{e}tale groupoid $\CG$ will be called a \emph{canonical} extended
metric. We will abuse notation and denote any such metric by $\rho$
or $\rho_{\CG}$. \end{defn}
\begin{example}
\label{exa:transformation-groupoid-length}Let $X$ be a compact Hausdorff space and $\Gamma$ be a countable group
that acts on $X$ by homeomorphisms. Then the transformation groupoid
$X\rtimes\Gamma$ (c.f., Example~\ref{exa:transformation-groupoid})
is $\sigma$-compact. To construct a coarse continuous length
function on $X\rtimes\Gamma$, we may fix a proper length function
$\ell_{\Gamma}$ on $\Gamma$
and then define 
\[
\ell_{X\rtimes\Gamma}:X\rtimes\Gamma\to[0,\infty),\quad(\gamma x,\gamma,x)\mapsto\ell_{\Gamma}(\gamma).
\]

We may also generalize this construction to the case where $X$ is locally compact and Hausdorff. To this end, we fix a continuous proper function $g \colon X \to [0, \infty)$ and define 
\begin{align*}
\ell_{X\rtimes\Gamma}(\gamma x,\gamma,x) := \inf_{\substack{n \in \mathbb{N}\\ (\gamma_n,  \ldots, \gamma_1, \gamma_0  ) \in \Gamma^{n+1} \colon \\ \gamma_n  \cdots \gamma_1 \gamma_0 = \gamma}} \sum_{k = 0}^n\ell_{\Gamma}(\gamma_k) \cdot (1 + g(\gamma_k  \cdots  \gamma_0 x) + g(\gamma_{k-1}  \cdots  \gamma_0 x)).
\end{align*}
It is straightforward to verify this is a continuous length function (one use the properness of $\ell_\Gamma$ to show this infimum is attained as a minimum, analogous to the proof of \Cref{thm:coarse-length-functions}). Properness follows from the inequalities $\ell_{X\rtimes\Gamma}(\gamma x,\gamma,x) \geq \ell_\Gamma(\gamma)$ and $\ell_{X\rtimes\Gamma}(\gamma x,\gamma,x) \geq m_\Gamma \cdot g(x)$ whenever $\gamma \not= 1_\Gamma$, where $m_\Gamma := \min_{\gamma \in \Gamma \setminus \{1_\Gamma\}} \ell_\Gamma(\gamma) > 0$. 
\end{example}
\begin{lem}
\label{lem:groupoid-uniformly-loc-finite}Any canonical extended metric
$\rho$ on a $\sigma$-compact \'{e}tale groupoid
$\CG$ is uniformly locally finite. \end{lem}
\begin{proof}
Let $\ell$ be the coarse length function that induces $\rho$ by
Lemma~\ref{lem:length-function-metric}. Fix $R>0$. Since $\ell$
is proper, the set $L_{R}=\left\{ z\in\CG\setminus\CG^{(0)}:\ell(z)\leq R\right\} $
is precompact. Then there is a finite family $\{V_{1},\dots,V_{m}\}$
of precompact open bisections such that 
\[
L_{R}\subset\bigcup_{i=1}^{m}V_{i}.
\]
Then, for every $x\in\CG$, one has 
\[
\bar{B}(x,R)=\{y\in G:\rho(y,x)\leq R\}=\{y\in\CG:\ell(yx^{-1})\leq R\}\subset\{x\}\cup L_{R}x\subset\{x\}\cup\bigcup_{i=1}^{m}V_{i}x.
\]
This implies that $|\bar{B}(x,R)|\leq m+1$. Since $R$ was arbitrarily
chosen and $m$ only depends on $R$, thus $\rho$ is uniformly locally
finite.\end{proof}
\begin{rem}
\label{rem:abstract-coarse-structure-groupoid}For readers familiar
with abstract coarse spaces in terms of entourages (c.f., \cite[Chapter 2]{Roe2003Lectures}),
we point out that the coarse structure on $\CG$ (as a set) determined
by any canonical extended metric can be directly defined as follows:
a subset $E$ of $\CG\times\CG$ is an entourage if and only if there
is a precompact subset $K$ of $\CG$ such that for any $(x,y)\in E$,
we have either $x=y$ or $x\in Ky$. This construction is different
from, but related to, the notion of coarse structures on a groupoid
(c.f., \cite{HigsonPedersenRoe1997C,TangWillettYao2018Roe}), in that
our entourages are subsets of $\CG\times\CG$ with $\CG$ embedded
as the diagonal, instead of subsets of $\CG$ with $\CG^{(0)}$ playing
the role of a diagonal, but on the other hand, our coarse structure
can be viewed as induced from the smallest coarse structure on the
groupoid $\CG$ (generated by the relatively compact subsets and $\CG^{(0)}$)
via the canonical translation action of $\CG$ on itself. 

A lot of the contents in this paper may be handled with this abstract
coarse structure, which would have the advantage of circumventing
the somewhat inconvenient fact that the canonical extended metric
is only unique up to coarse equivalence. The definition and some basic
properties even extend beyond the case of $\sigma$-compact \'{e}tale groupoids. However, we opt to stick to the language of metrics
since it is more intuitive while $\sigma$-compact \'{e}tale groupoids
are prevalent in the main applications we have in mind (indeed, it
is necessary to ensure that $C_{r}^{*}(\CG)$ is separable). 
\end{rem}

\section{Fiberwise amenability \label{sec:fiberwise-amenability}}
In this section, we will introduce  fiberwise amenability
and ubiquitous fiberwise amenability for \'{e}tale groupoids, with inspirations
from (ubiquitous) metric amenability (c.f., Definition~\ref{def:metric-amenability}
and~\ref{def:uniform-metric-amenability}). Fiberwise amenability
is closely related to the existence of invariant measures on the unit
space of an \'{e}tale groupoid. As a motivating example, a transformation
groupoid is (ubiquitously) fiberwise amenable if and only if the acting
group is amenable. Ubiquitous fiberwise amenability will play an important
auxiliary role when we discuss groupoid strict comparison and almost
elementariness for minimal groupoids in our sequel paper \cite{MWAE}. For this
purpose, we show in \Cref{subsec:minimal} that for minimal
groupoids, fiberwise amenability is also equivalent to the a priori
stronger notion of ubiquitous fiberwise amenability. 

We first define boundary sets in groupoids in analogy with Definition~\ref{def:metric-boundaries}. 
\begin{defn}\label{def:groupoid-boundaries}
Let $\CG$ be a groupoid. For any subsets $A,K\subseteq\CG$, we define
the following boundary sets:
\begin{enumerate}[label=(\roman*)]
\item \emph{left outer $K$-boundary}: $\partial_{K}^{+}A=(KA)\setminus A=\{yx\in\CG\setminus A:y\in K,x\in A\}$;
\item \emph{left inner $K$-boundary}: $\partial_{K}^{-}A=A\cap(K^{-1}(\CG\setminus A))=\{x\in A:yx\in\CG\setminus A\mbox{ for some }y\in K\}$;
\item \emph{left $K$-boundary}: $\partial_{K}A=\partial_{K}^{+}A\cup\partial_{K}^{-}A$. 
\end{enumerate}
\end{defn}
Observe that if $A$ as above is contained in a single source fiber,
then $KA$ and all these boundary sets are also contained in this
source fiber. This is the reason for the terminology \emph{fiberwise
amenability}. 
\begin{rem}\label{rem: double control for boundries}
\label{3.1-1} For any subsets $A,K\subseteq\CG$, it is straightforward
to see $\partial_{K}^{+}A\subset K\partial_{K}^{-}A$ and $\partial_{K}^{-}A\subset K^{-1}\partial_{K}^{+}A$. 
\end{rem}
The following concept is analogous to the metric case, too. 
\begin{defn}
\label{def:groupoid-Folner} Let $\CG$ be an \'{e}tale
groupoid. For any subset $K\subseteq\CG$ and $\varepsilon>0$, a finite
non-empty set $F\subset \CG$ is called $(K,\varepsilon)$-F{\o}lner if
it satisfies 
\[
\frac{|\partial_{K}F|}{|F|}\leq\varepsilon.
\]
We denote by $\opFol(K,\varepsilon)$ the collection of all $(K,\varepsilon)$-F{\o}lner
sets. 
\end{defn}
This leads to a natural definition of fiberwise amenability. 
\begin{defn}
\label{def:fiberwise-amenable}\label{4.1} Let $\CG$ be a locally
compact \'{e}tale groupoid. 
\begin{enumerate}
\item\label{def:fiberwise-amenable:original} We say $\CG$ is \textit{fiberwise amenable} if for any compact subset
$K$ of $\CG$ and any $\varepsilon>0$, there exists a $(K,\varepsilon)$-F{\o}lner
set.
\item\label{def:fiberwise-amenable:ubiquitous} We say $\CG$ is \textit{ubiquitously fiberwise amenable} if and only
if for any compact subset $K$ of $\CG$ and any $\varepsilon>0$, there
exists a compact subset $L$ of $\CG$ such that for any unit $u\in\CG^{(0)}$,
there is a $(K,\varepsilon)$-F{\o}lner set in $Lu\cup\{u\}$. 
\end{enumerate}
\end{defn}
Since the groupoids we focus on are $\sigma$-compact and come equipped
with extended metric structure in a somewhat canonical way (c.f.,
Definition~\ref{def:coarse-metric}), we may also reformulate Definition~\ref{def:fiberwise-amenable}
using the canonical extended metric. This will establish a connection
with Section~\ref{sec:metric-amenability} and enable us to apply
the results there. 
\begin{prop}
\label{prop:fiberwise-amenability-metric} Let $\CG$ be a $\sigma$-compact
\'{e}tale groupoid and let $(\CG,\rho)$ be
the extended metric space induced by a coarse length function $\ell$. 
\begin{enumerate}
\item
\label{prop:fiberwise-amenability-metric::fiberwise-amenable} The groupoid $\CG$ is fiberwise amenable if and only if for any compact
subset $K$ of $\CG$ and any $\varepsilon>0$, there exists a nonempty
finite subset $F$ in $\CG$ satisfying 
\[
\frac{|KF \cup F|}{|F|}\leq1+\varepsilon,
\]
if and only if $(\CG,\rho)$ is amenable in the sense of Definition
\ref{3.2}.. 
\item
\label{prop:fiberwise-amenability-metric::ubiquitously-fiberwise-amenable} The groupoid $\CG$ is ubiquitously fiberwise amenable if and only
if for any compact subset $K$ of $\CG$ and any $\varepsilon>0$, there
exists a compact subset $L$ of $\CG$ such that for any unit $u\in\CG^{(0)}$,
there is a nonempty finite subset $F$ in $Lu\cup\{u\}$ satisfying
\[
\frac{|KF \cup F|}{|F|}\leq1+\varepsilon,
\]
if and only if $(\CG,\rho)$ is ubiquitously amenable in the sense
of Definition \ref{3.0}. 
\end{enumerate}
\end{prop}
\begin{proof}
We prove the second statement, the first being similar. To prove the
three conditions here are equivalent, we first observe that the equivalence
of the first two follows from Remark~\ref{3.1-1} and the fact that if $K$ is compact in $\CG$ then there is an $m>0$ such that $|K^{-1}F|\leq m|F|$ for any finte set $F$ in $\CG$ (note that this does not require $\sigma$-compactness). The equivalence between the first two conditions and the ubiquitous amenable of $(\CG,\rho)$ follows directly from Proposition~\ref{prop:metric-amenable-change-boundaries},  Lemma~\ref{lem:invariant-fiberwise-extended-metric}, and the observation that since continuous length functions are controlled (see Definition~\ref{def:coarse-length-function}), thus by enlarging the set $K$ if necessary, we may assume without loss of generality that $K = \bar{B}\left(\GU,r\right)$ for some $r \geq 0$. 
\end{proof}
\begin{rem}
\label{4.2}\label{rem:transformation-groupoid-fiberwise-amenable}
Let $\alpha:\Gamma\curvearrowright X$ be an action of a countable
discrete group $\Gamma$ on a compact space $X$. We denote by $X\rtimes_{\alpha}\Gamma$
the transformation groupoid of this action $\alpha$. When we equip
$\Gamma$ with a proper length function $\ell_{\Gamma}$ and $X\rtimes_{\alpha}\Gamma$
with the induced length function $\ell_{X\rtimes_{\alpha}\Gamma}:(\gamma x,\gamma,x)\mapsto\ell_{\Gamma}(\gamma)$
(c.f., Example~\ref{exa:transformation-groupoid-length}), each source
fiber $(X\rtimes_{\alpha}\Gamma)_{x}=\{(\gamma x,\gamma,x):\gamma\in\Gamma\}$,
for $x\in X$, becomes isometric to $\Gamma$. Therefore $X\rtimes_{\alpha}\Gamma$
is fiberwise amenable if and only if $\Gamma$ is amenable. 

In particular, the group $\Gamma$ as a groupoid is fiberwise amenable
if and only if $\Gamma$ is amenable. In addition, it follows from
the homogeneity of a group to see that $\Gamma$ is amenable if and
only if it is ubiquitously fiberwise amenable . 
\end{rem}

\begin{rem}
\label{rem:fiberwise-amenability-noncompact}Fiberwise amenability
is not an interesting property for groupoids $\CG$ with noncompact
unit spaces, for it is automatically satisfied in this case. Indeed,
for any compact subset $K$ of $\CG$, if we choose an arbitrary point
$u$ in $\CG^{(0)}\setminus s(K)$, then $Ku=\varnothing$ and thus
$\{u\}$ becomes a $(K,0)$-F{\o}lner set. Ubiquitous fiberwise amenability
may still fail; an easy example being the disjoint union of two groupoids,
the first having a noncompact unit space and the second lacking ubiquitous
fiberwise amenability. 
\end{rem}

Focusing on the case of compact unit spaces, we next show that fiberwise
amenability implies the existence of invariant probability measures
on unit spaces. This directly generalizes the case of actions by amenable
groups on compact spaces. 
\begin{defn}
\label{def:invariant-measure}A measure on the unit space of an \'{e}tale groupoid $\CG$ is \emph{invariant} if $\mu(r(U))=\mu(s(U))$
for any Borel measurable bisection $U$. We write $M(\CG)$ for the collection
of all \emph{invariant regular Borel probability measures} on $\CG^{(0)}$. \end{defn}

Let $f\in C_c(\CG)_+$ be supported on an open bisection $U$. For simplicity, we define $s(f)=(f^**f)^{1/2}$ and $r(f)=(f*f^*)^{1/2}$, which are functions supported on $s(U)$ and $r(U)$, respectively. Note that by definition, $s(f)(u)=f((s|_{U})^{-1}(u))$ for $u\in s(U)$ and $r(f)(u)=f((r|_{U})^{-1}(u))$ for $u\in r(U)$.

\begin{prop}
	\label{4.12} 
	\label{prop:fiberwise-amenable-invariant-measure}
	Let $\CG$ be 
	an \'{e}tale groupoid with a compact unit space.
	Let $\{K_n \colon n\in \N\}$ be an increasing sequence of compact subsets in $\CG$ such that $\bigcup_{n\in \N}K_n^{\operatorname{o}}=\CG$. Let $\{\varepsilon_{n} \colon n\in\N\}$ be a decreasing sequence of positive numbers converging to $0$. Let $\{F_n \colon n\in \N\}$ be a sequence of finite subsets in $\CG$ such that  $|K_nF_n|<(1+\varepsilon_{n})|F_{n}|$ for any $n \in \N$. Then in the space of Borel probability measures on $\GU$, any $w^{\ast}$-cluster point $\mu$ of the sequence 
	\[
		\mu_{n}=\frac{1}{|F_{n}|}\sum_{x\in F_{n}}\delta_{r(x)} \, , \quad n = 0, 1, 2, \ldots
	\]
	of the probability measures belongs to $M(\CG)$. 
\end{prop}

\begin{proof}
	It suffices to show that $\mu(r(f))=\mu(s(f))$ for any function
	$f\in C_{c}(\CG)_{+}$ whose support $\spp(f)$ is a compact bisection.
	We may also assume $\|f\|\leq1$. Write $K=\spp(f)$ for simplicity.
	Note that $r(f),s(f)$ are functions supported on $r(K)$ and $s(K)$,
	respectively.

	By passing to subsequences, we may assume without loss of generality that $\mu_{n}\rightarrow\mu$ in the $w^{\ast}$-topology. 
	We show that $\mu\in M(\CG)$ by estimating
	$|\mu(r(f))-\mu(s(f))|$: 
	\[
	|\mu(r(f))-\mu(s(f))|\leq|\mu(r(f))-\mu_{n}(r(f))|+|\mu_{n}(r(f))-\mu_{n}(s(f))|+|\mu(s(f))-\mu_{n}(s(f))|.
	\]
    Focuing on the middle term, we have 
	\begin{align*}
	|\mu_{n}(r(f))-\mu_{n}(s(f))| & = \left|\frac{1}{|F_{n}|}\left(\sum_{x\in F_{n}}r(f)(r(x))-\sum_{x\in F_{n}}s(f)(r(x))\right)\right|\\
	& =\left|\frac{1}{|F_{n}|}\left(\sum_{\substack{x\in F_n\\r(x)\in r(K)}}r(f)(r(x))-\sum_{\substack{x\in F_n\\r(x)\in s(K)}}s(f)(r(x))\right)\right|\\
	& =\left|\frac{1}{|F_{n}|}\left(\sum_{x\in r(K)F_{n}}r(f)(r(x))-\sum_{x\in KF_{n}}r(f)(r(x))\right)\right|\\
	& \leq \frac{1}{|F_{n}|}\sum_{x\in r(K)F_{n} \mathrel{\Delta} KF_{n}} \left|r(f)(r(x))\right| \\
    & \leq \frac{|r(K)F_{n}\mathrel{\Delta} KF_{n}|}{|F_{n}|}.
	\end{align*}
    Recalling notations from \Cref{def:groupoid-boundaries}, we observe that 
    \begin{align*}
        r(K)F_{n}\setminus KF_{n} &= \left\{ x \in F_{n} \colon r(x) \in r(K) \text{ and } K^{-1} x \cap F_{n} =\varnothing \right\} \\
        &= \left\{ x \in F_{n} \colon yx \notin F_{n} \text{ for some } y\in K^{-1}  \right\} = \partial_{K^{-1}}^-F_n \;, \quad \text{ and } \\
        K F_{n}\setminus r(K)F_{n} &= K F_{n}\setminus F_{n} = \left\{ yx \colon  x \in F_{n} ,\ y \in K , \ \text{ and } xy \notin F_{n}  \right\} = \partial_{K}^+F_n \; .
    \end{align*}
    Applying \Cref{rem: double control for boundries}, we see that $r(K)F_{n}\mathrel{\Delta} KF_{n} \subseteq \partial_{K }^+ F_n \cup K \partial_{K^{-1}}^+ F_n$. Since $K$ is a bisection, we have 
    \begin{align*}
        \frac{|r(K)F_{n}\mathrel{\Delta} KF_{n}|}{|F_{n}|} \leq 2 \frac{|\partial_{K \cup K^{-1}}^+ F_n|}{|F_{n}|} = 2 \frac{|(K \cup K^{-1}) F_n \cup F_{n}| - | F_{n}|}{|F_{n}|} \; .
    \end{align*}
    
	Now for every $\varepsilon>0$, we choose an integer $n$ large enough such that $K\cup K^{-1} \cup \mathcal{G}^{(0)}\subset K_n$ (which is possible by a compactness argument), $\varepsilon_n\leq \varepsilon/6$, $|\mu(r(f))-\mu_n(r(f))|<\varepsilon/3$ and $|\mu(s(f))-\mu_n(s(f))|<\varepsilon/3$. This implies that
	\[|\mu(r(f))-\mu(s(f))|< \varepsilon/3 + \varepsilon/3 + \varepsilon/3 =\varepsilon.\]
	This establishes $\mu(s(f))=\mu(r(f))$ as desired.
\end{proof}

\begin{cor}
\label{cor:fiberwise-amenable-invariant-measure}
Let $\CG$ be a fiberwise amenable, $\sigma$-compact,
\'{e}tale groupoid with a compact unit space.
Then $M(\CG)\neq\emptyset$.
\end{cor}

\begin{proof}
	This follows directly from Proposition~\ref{prop:fiberwise-amenable-invariant-measure}. 
\end{proof}

We remark that the converse of the above result is not true. For example, Hjorth and Molberg showed in \cite{HjorthMolberg2006Free}  that all countable discerete group admits a free action on Cantor set  with a probability invariant measure. This means there is a  transformation groupoid $\CG$ of a certain action of a free group  whose $M(\CG)$ is not empty (a concrete example of this kind is given by the odometer action of the free group $\mathbb{F}_2$ on one of its profinite completions). However, Remark \ref{rem:transformation-groupoid-fiberwise-amenable} shows that such $\CG$ is not fiberwise amenable.

\subsection{The case of minimal groupoids}\label{subsec:minimal}

In this subsection, we show that for minimal groupoids, fiberwise
amenability is also equivalent to the a priori stronger notion of ubiquitous
fiberwise amenability. The strategy to show the former implies the
latter, roughly speaking, is: on the one hand, a F{\o}lner set on a single
source fiber, is always able to ``permeate'' horizontally to nearby
fibers; on the other hand, the recurrence behavior guaranteed by minimality
allows every source fiber to ``pick up'' a F{\o}lner set from this
permeation every so often, thus resulting in ubiquitous fiberwise
amenability.

To explain how this ``permeation'' arises, it is convenient to use
the following result about the existence of local trivializations
that almost preserve the metric. 
\begin{lem}[\emph{Local Slice Lemma}]
\label{lem:local-slice}Let $\CG$ be a $\sigma$-compact \'{e}tale groupoid and let $\rho$ be a canonical extended
metric on $\CG$ induced by a coarse continuous length function $\ell$
as in Definition~\ref{def:coarse-metric}. Let $u\in\CG^{(0)}$.
Then for any $R,\eps>0$, there are a number $S\in[R,R+\eps)$, an
open neighborhood $V$ of $u$ in $\CG^{(0)}$, an open set $W$ in
$\CG$, and a homeomorphism $f:\bar{B}_{\rho}(u,S)\times V\to W$
such that 
\begin{enumerate}
\item $f(u,v)=v$ for any $v\in V$, 
\item $f\left(x,u\right)=x$ for any $x\in\bar{B}_{\rho}(u,S)$,
\item $f\left(\bar{B}_{\rho}(u,S)\times\{v\}\right)=\bar{B}_{\rho}(v,S)$
for any $v\in V$, and
\item $\left|\rho\left(x,y\right)-\rho\left(f(x,v),f(y,v)\right)\right|<\eps$
for any $x,y\in\bar{B}_{\rho}(u,S)$ and $v\in V$. 
\end{enumerate}
\end{lem}
\begin{proof}
By Lemma~\ref{lem:groupoid-uniformly-loc-finite}, the ``open''
ball $B_{\rho}(u,R+\eps)$, i.e., the set $\left\{ x\in\CG_{u}:\ell(x)<R+\eps\right\} $,
is finite, and thus 
\[
\overline{\ell}\left(B_{\rho}(u,R+\eps)\right)=\max\left\{ \ell(x):x\in B_{\rho}(u,R+\eps)\right\} <R+\eps.
\]
Hence we may choose $S\in[R,R+\eps)\cap\left(\overline{\ell}\left(B_{\rho}(u,R+\eps)\right),R+\eps\right)$,
e.g., 
\[
S=\max\left\{ R,\frac{\overline{\ell}\left(B_{\rho}(u,R+\eps)\right)+R+\eps}{2}\right\} ,
\]
which guarantees $\bar{B}_{\rho}(u,S)=B_{\rho}(u,R+\eps)$ and thus
$S>\overline{\ell}\left(\bar{B}_{\rho}(u,S)\right)$. 

For each $x\in\bar{B}_{\rho}(u,S)$, choose an open bisection $U_{x}$
containing $x$ and let $f_{x}:s\left(U_{x}\right)\to U_{x}$ be the
inverse of the homeomorphism $s\mid_{U_{x}}$. Without loss of generality,
we may assume $U_{u}=\CG^{(0)}$ and $f_{u}$ is the identity map.
Define 
\[
L=\ell^{-1}([0,S])\setminus\left(\bigcup_{x\in\bar{B}_{\rho}(u,S)}U_{x}\right)\quad\mbox{and}\quad U=\CG^{(0)}\setminus s(L).
\]
Unpacking the definition and using the fact $\bar{B}_{\rho}(v,S)=\ell^{-1}([0,S])\cap s^{-1}(v)$
for any $v\in\CG^{(0)}$, we have 
\begin{equation}
U=\left\{ v\in\CG^{(0)}:\bar{B}_{\rho}(v,S)\subseteq\bigcup_{x\in\bar{B}_{\rho}(u,S)}U_{x}\right\} \label{eq:rewriting-U}
\end{equation}
and, in particular, $u\in U$. Since $\ell$ is proper and continuous as well as $L\cap \GU=\emptyset$,
we see that $L$ is compact and hence $U$ is an open neighborhood
of $u$ in $\CG^{(0)}$. Define a continuous map 
\[
f:\bar{B}_{\rho}(u,S)\times U\to\CG,\quad(x,v)\mapsto f_{x}(v).
\]
It follows from the construction of the $f_{x}$'s that 
\begin{enumerate}
\item $f(u,v)=v$ for any $v\in U$, and
\item $f\left(x,u\right)=x$ for any $x\in\bar{B}_{\rho}(u,S)$. 
\end{enumerate}
We also have $\left(s\circ f\right)\left(x,v\right)=v$ for any $(x,v)\in\bar{B}_{\rho}(u,S)\times U$.
It then follows from \eqref{eq:rewriting-U} that 
\begin{equation}
f\left(\bar{B}_{\rho}(u,S)\times\{v\}\right)\supseteq\bar{B}_{\rho}(v,S)\quad\mbox{for any }v\in U.\label{eq:f-surjective-half}
\end{equation}

Now we define a finite collection of continuous maps
\[
g_{xy}:U\to[0,\infty),\quad v\mapsto\rho\left(f(x,v),f(y,v)\right),
\]
for $x,y\in\bar{B}_{\rho}(u,S)$ and define 
\[
\eta=\min\left\{ \eps,S-\overline{\ell}\left(\bar{B}_{\rho}(u,S)\right),\frac{\rho(x,y)}{2}:x,y\in\bar{B}_{\rho}(u,S)\mbox{ with }x\not=y\right\} .
\]
Note that $\eta>0$ by our choice of $S$. By continuity, there exists
an open neighborhood $V$ of $u$ inside $U$ such that $\left|g_{xy}(u)-g_{xy}(v)\right|<\eta$
for any $v\in V$ and $x,y\in\bar{B}_{\rho}(u,S)$. This choice implies
the following: 
\begin{enumerate}
\item[(3)] For any $x,y\in\bar{B}_{\rho}(u,S)$ and $v\in V$, since $g_{xy}(u)=\rho(x,y)$,
we have 
\[
\left|\rho\left(x,y\right)-\rho\left(f(x,v),f(y,v)\right)\right|<\eps.
\]

\item[(4)] For any $x\in\bar{B}_{\rho}(u,S)$ and $v\in V$, we have $\ell(f(x,v))=\rho\left(f(x,v),f(u,v)\right)=g_{xu}(v)<g_{xu}(u)+\eta=\ell(x)+\eta\leq\overline{\ell}\left(\bar{B}_{\rho}(u,S)\right)+\eta\leq S$,
and thus combined with \eqref{eq:f-surjective-half}, we have 
\[
f\left(\bar{B}_{\rho}(u,S)\times\{v\}\right)=\bar{B}_{\rho}(v,S).
\]
\end{enumerate}
\begin{itemize}
\item For any $v\in V$ and any $x,y\in\bar{B}_{\rho}(u,S)$ with $x\not=y$,
we have $\rho\left(f(x,v),f(y,v)\right)=g_{xy}(v)>g_{xy}(u)-\eta=\rho(x,y)-\eta>0$
and thus $f(x,v)\not=f(y,v)$. This implies that the collection $\left\{ f_{x}(V):x\in\bar{B}_{\rho}(u,S)\right\} $
of open sets is disjoint and $f$ is a homeomorphism onto its image
when restricted to $\bar{B}_{\rho}(u,S)\times V$. 
\end{itemize}
Defining $W=f\left(\bar{B}_{\rho}(u,S)\times V\right)$ and restricting
$f$ to $\bar{B}_{\rho}(u,S)\times V$ thus completes the construction. 
\end{proof}
The existence of local slices as in Lemma~\ref{lem:local-slice}
allows us to ``clone'' a F{\o}lner set in every nearby source fiber. 
\begin{lem}
	\label{lem:Folner-permeates}Let $\CG$ be a $\sigma$-compact \'{e}tale groupoid and let $\rho$ be a canonical extended
	metric on $\CG$ induced by a coarse continuous length function $\ell$
	as in Definition~\ref{def:coarse-metric}. Let $R,S,\eps>0$ with $R>\eps$ and
	$u\in\CG^{(0)}$. Then there is an open neighborhood $V$ of $u$
	in $\CG^{(0)}$ such that whenever there exist $v_{0}\in V$ and an
	$(R,\eps)$-F{\o}lner set in $\bar{B}_{\rho}\left(v_{0},S\right)$, then
	for any $v\in V$, there is an $(R-\eps,\eps)$-F{\o}lner set in $\bar{B}_{\rho}\left(v,S+\varepsilon\right)$. \end{lem}
\begin{proof}
	Let $R'=S+R+\eps$ and 
	\[
	\eta=\frac{1}{2}\min\left\{ \eps,\rho(x,y)-R+\eps:x,y\in\bar{B}_{\rho}(u,R'+\eps)\mbox{ with }\rho(x,y)>R-\eps\right\} .
	\]
	Applying Lemma~\ref{lem:local-slice} with $u$, $R'$, and $\eta$
	in place of $u$, $R$, $\eps$, we obtain a number $S'\in[R',R'+\eta)$,
	an open neighborhood $V$ of $u$ in $\CG^{(0)}$, an open set $W$
	in $\CG$, and a homeomorphism $f:\bar{B}_{\rho}(u,S')\times V\to W$
	such that 
	\begin{enumerate}
		\item $f(u,v)=v$ for any $v\in V$, 
		\item $f\left(x,u\right)=x$ for any $x\in\bar{B}_{\rho}(u,S')$, 
		\item $f\left(\bar{B}_{\rho}(u,S')\times\{v\}\right)=\bar{B}_{\rho}(v,S')$
		for any $v\in V$, and 
		\item $\left|\rho\left(x,y\right)-\rho\left(f(x,v),f(y,v)\right)\right|<\eta$
		for any $x,y\in\bar{B}_{\rho}(u,S')$ and $v\in V$. 
	\end{enumerate}
	Now assuming that there is an $(R,\eps)$-F{\o}lner set $F$ in $\bar{B}_{\rho}\left(v_{0},S\right)$,
	we then define, for any $v\in V$, the bijections 
	\[
	\tau_{v}:\bar{B}_{\rho}(u,S')\to\bar{B}_{\rho}(v,S'),\quad x\mapsto f(x,v)
	\]
	and the set 
	\[
	F_{v}=\tau_{v}\circ\tau_{v_{0}}^{-1}\left(F\right).
	\]
	For any $v\in V$, we claim that $F_{v}$ is the desired $(R,\eps)$-F{\o}lner
	set in $\bar{B}_{\rho}\left(v,S+\eps\right)$. Indeed, it follows
	from condition~(4) that 
	\[
	\overline{\ell}\left(F_{v}\right)<\overline{\ell}\left(\tau_{v_{0}}^{-1}\left(F\right)\right)+\eta\leq\overline{\ell}(F)+2\eta\leq S+\eps
	\]
	and thus $F_{v}\subseteq\bar{B}_{\rho}\left(v,S+\eps\right)$. On
	the other hand, to see $F_{v}$ is an $(R,\eps)$-F{\o}lner set just
	like $F$, it suffices to show that 
	\[
	\partial_{R-\varepsilon}^{+}F_{v}\subseteq\tau_{v}\circ\tau_{v_{0}}^{-1}\left(\partial_{R}^{+}F\right)\quad\mbox{and}\quad\partial_{R-\varepsilon}^{-}F_{v}\subseteq\tau_{v}\circ\tau_{v_{0}}^{-1}\left(\partial_{R}^{-}F\right).
	\]

	To prove the former containment, we observe that for any $y\in\partial_{R-\eps}^{+}F_{v}$,
	since $\ell(y)\leq\overline{\ell}\left(F_{v}\right)+R-\eps\leq S+R< R'\leq S'$,
	it is in the range of $\tau_{v}$. Let $x=\tau_{v_{0}}\circ\tau_{v}^{-1}(y)$.
	Since $y\not\in F_{v}$, we have $x\not\in F$. It remains to show
	that $\rho(x,F)\leq R$. Suppose this were not the case, i.e., for
	any $z\in F$, we have $\rho(x,z)>R$. First note that $\ell(\tau_{v_{0}}^{-1}(x))=\ell(\tau_{v}^{-1}(y))\leq \ell(y)+\eta\leq R'$ and  $\ell(\tau_{v_{0}}^{-1}(z))\leq \ell(z)+\eta\leq S+\eta<R'$ for any $z\in F$. Then by our choice of $\eta$,
	we would have 
	\begin{align*}
	\rho(y,F_{v})&=\min\left\{ \rho\left(\tau_{v}\circ\tau_{v_{0}}^{-1}(x),\tau_{v}\circ\tau_{v_{0}}^{-1}(z)\right):z\in F\right\}\\
	 &>\min\left\{ \rho\left(\tau_{v_{0}}^{-1}(x),\tau_{v_{0}}^{-1}(z)\right)-\eta:z\in F\right\} \geq R-\varepsilon,
	\end{align*}
	contradictory to the fact that $y\in\partial_{R-\varepsilon}^{+}F_{v}$. This
	shows $\partial_{R-\eps}^{+}F_{v}\subseteq\tau_{v}\circ\tau_{v_{0}}^{-1}\left(\partial_{R}^{+}F\right)$.
	
	To prove the latter containment, we observe that any $y\in\partial_{R-\eps}^{-}F_{v}$
	is in $F_{v}$ and thus we may define $x=\tau_{v_{0}}\circ\tau_{v}^{-1}(y)$
	in $F$. It remains to show that $\rho(x,\CG\setminus F)\leq R$.
	Suppose this were not the case. Note that $\ell(y)\leq \overline{\ell}\left(F_{v}\right)\leq S+\eps$. In addition, $\ell(\tau_{v_{0}}^{-1}(x))\leq S+\eta<R'$ and $\ell(\tau_{v_{0}}^{-1}(z))\leq S'+\eta<R'+\varepsilon$ for any $z\in \bar{B}_{\rho}(v_0,S')$. Then by the decomposition $\CG\setminus F_{v}=\left(\CG\setminus\bar{B}_{\rho}(v,S')\right)\cup\left(\bar{B}_{\rho}(v,S')\setminus F_{v}\right)$
	and our choice of $S$ and $\eta$, we would have 
	\begin{multline*}
		\rho(y,\CG\setminus F_{v})=\inf\left\{ \rho\left(y,w\right):w\in\CG\setminus F_{v}\right\} \\
		=\inf\left\{ \rho\left(y,w\right),\rho\left(\tau_{v}\circ\tau_{v_{0}}^{-1}(x),\tau_{v}\circ\tau_{v_{0}}^{-1}(z)\right):w\in \CG\setminus\bar{B}_{\rho}(v,S'),z\in\bar{B}_{\rho}(v_0,S')\setminus F\right\} \\
		>\min\left\{ S'-\ell(y),\rho\left(\tau_{v_{0}}^{-1}(x),\tau_{v_{0}}^{-1}(z)\right)-\eta :z\in\bar{B}_{\rho}(v_0,S')\setminus F\right\} \geq R-\eps,
	\end{multline*}
	contradictory to the fact that $y\in\partial_{R-\eps}^{-}F_{v}$. This
	shows $\partial_{R-\varepsilon}^{-}F_{v}\subseteq\tau_{v}\circ\tau_{v_{0}}^{-1}\left(\partial_{R}^{-}F\right).$
	and completes the proof. 
\end{proof}

The following lemma underlies the recurrence behavior of minimal groupoids
with compact unit spaces. 
\begin{lem}
\label{4.9} \label{lem:minimal-recurrence}
Let $\CG$ be an
\'{e}tale groupoid. Let $K$ and $V$ be subsets
of $\CG^{(0)}$ such that $K$ is compact, $V$ is open, and $K \subseteq r \left( \CG \cdot V \right)$. Then there are compact bisections $K_{1},\dots,K_{n}$
such that $\bigcup_{i=1}^{n}r\left(K_{i}\right)\subseteq V$ and $K\subseteq\bigcup_{i=1}^{n}s\left(K_{i}\right)^{\operatorname{o}}$. 
\end{lem}

Note that the condition $K \subseteq r \left( \CG \cdot V \right)$ holds automatically when $\CG$ is minimal and $V$ is nonempty. 

\begin{proof}
Since $K \subseteq r \left( \CG \cdot V \right)$, for any $u\in K$, there is a $v_u \in V$ and
$x_u \in \CG$ such that $r(x_u)=v_u$ and $s(x_u)=u$. Then since $\CG$ is
locally compact, Hausdorff and \'{e}tale, there is a compact bisection $K_{u}$
such that $x_u \in K_u^{\operatorname{o}}$ and $r \left( K_u \right) \subseteq V$. 
This implies that 
and $u=s(x_u)\in s\left(K_{u}^{\operatorname{o}}\right) = s\left(K_{u}\right)^{\operatorname{o}}$. 
Hence we have an open cover $\left\{ s\left(K_{u}\right)^{\operatorname{o}} \colon u \in K \right\}$ of $K$. 
By compactness, there are finitely many
compact bisections $K_{1},\dots,K_{n}$ such that $K\subseteq\bigcup_{i=1}^{n}s\left(K_{i}\right)^{\operatorname{o}}$.
In addition, our construction implies $\bigcup_{i=1}^{n}r\left(K_{i}\right)\subseteq V$. 
\end{proof}

Now we are ready to establish the equivalence of fiberwise amenability
and ubiquitously fiberwise amenability for minimal groupoids.
\begin{thm}
\label{4.01} \label{thm:minimal-fiberwise-amenability}Let $\CG$
be a $\sigma$-compact \'{e}tale groupoid. Suppose
$\CG$ is minimal. Then $\CG$ is fiberwise amenable if and only
if it is ubiquitously fiberwise amenable. \end{thm}
\begin{proof}
The ``if'' direction follows directly from the definitions. To show
the ``only if'' direction, we let $\rho$ be a canonical extended
metric on $\CG$ induced by a coarse continuous length function $\ell$
as in Definition~\ref{def:coarse-metric}. By Proposition~\ref{prop:fiberwise-amenability-metric},
it suffices to show, assuming the extended metric space $(\CG,\rho$)
is amenable, that it is also ubiquitously amenable, i.e., for every
$R>0$ and $\eps>0$, there exists an $S>0$ such that for any $x\in\CG$,
there is an $(R,\eps)$-F{\o}lner set $F$ in the ball $\bar{B}_{\rho}(x,S)$.
To this end, given $R,\eps>0$, since we assume $(\CG,\rho$) is amenable,
we know there exists an $(R+\eps,\eps)$-F{\o}lner set $F_{0}$ in $\CG$.
By Lemma~\ref{lem:Folner-components}, we may assume without loss
of generality that $F_{0}$ is contained in a single source fiber
$\CG_{u}$ for some $u\in\CG^{(0)}$. By Lemma~\ref{lem:Folner-permeates},
there is an open neighborhood $V$ of $u$ in $\CG^{(0)}$ such that
for any $v\in V$, there is an $(R,\eps)$-F{\o}lner set $F_{v}$ in
$\bar{B}_{\rho}\left(v,\overline{\ell}\left(F_{0}\right)+\eps\right)$.
Let 
\[
K=s\left(\ell^{-1}([0,R])\setminus\CG^{(0)}\right),
\]
which is a compact subset of $\CG^{(0)}$, as $\ell$ is a continuous
proper length function. By Lemma~\ref{lem:minimal-recurrence}, there
are precompact open bisections $V_{1},\dots,V_{n}$ such that $\bigcup_{i=1}^{n}r\left(V_{i}\right)\subseteq V$
and $K\subseteq\bigcup_{i=1}^{n}s\left(V_{i}\right)$. Let 
\[
S=\overline{\ell}\left(F_{0}\right)+\eps+\max\left\{ \overline{\ell}\left(V_{i}\right):i=1,\ldots,n\right\} ,
\]
which is finite since all the sets involved are precompact. Now, for
any $x\in\CG$, we need to construct an $(R,\eps)$-F{\o}lner set $F$
in $\bar{B}_{\rho}(x,S)$. There are two cases:
\begin{itemize}
\item If $r(x)\not\in K$, then $\bar{B}_{\rho}(x,R)=\bar{B}_{\rho}(r(x),R)x=\left(\CG_{r(x)}\cap\ell^{-1}([0,R])\right)x=\left\{ r(x)x\right\} =\{x\}$
by our choice of $K$, and thus we may set $F=\{x\}$, which is an
$(R,0)$-F{\o}lner set. 
\item If $r(x)\in K$, then we may choose $i_{x}\in\{1,\ldots,n\}$ such
that $r(x)\in s\left(V_{i_{x}}\right)$. Let $z\in V_{i_{x}}$ be
such that $r(x)=s(z)$. Note that $r(z)\in V$ and thus we have an
$(R,\eps)$-F{\o}lner set $F_{r(z)}$ in $\bar{B}_{\rho}\left(r(z),\overline{\ell}\left(F_{0}\right)+\eps\right)$.
Let $F=F_{r(z)}zx$, which is also an $(R,\eps)$-F{\o}lner set by the
right-invariance of $\rho$ (see Lemma~\ref{lem:length-function-metric}).
Finally, since for any $y\in F_{r(z)}$, we have $\rho(yzx,x)=\rho(yz,r(x))=\ell(yz)\leq\overline{\ell}\left(F_{r(z)}\right)+\overline{\ell}\left(V_{i_{x}}\right)\leq\overline{\ell}\left(F_{0}\right)+\eps+\overline{\ell}\left(V_{i_{x}}\right)\leq S$,
we conclude that $F$ is in the ball $\bar{B}_{\rho}(x,S)$. 
\end{itemize}
This completes the proof. 
\end{proof}

\begin{cor}
Let $\CG$ be a $\sigma$-compact \'{e}tale
groupoid. Suppose $\CG$ is minimal and $\CG^{(0)}$ is noncompact.
Then $\CG$ is ubiquitously fiberwise amenable. \end{cor}
\begin{proof}
This follows from Theorem~\ref{thm:minimal-fiberwise-amenability}
and Remark~\ref{rem:fiberwise-amenability-noncompact}. 
\end{proof}

\subsection{Coarse groupoids}\label{subsec:coarse_groupoid}

In this subsection, we indicate a further connection between
metric amenability discussed in \Cref{sec:metric-amenability} and
fiberwise amenability via the construction of coarse groupoids. 
\begin{defn}[{\cite[Proposition 3.2]{SkandalisTuYu2002coarse}}]
	\label{def:coarse-groupoid}Let $(Y,d)$ be a uniformly locally finite
	extended metric space. The \emph{coarse groupoid} $\CG_{(Y,d)}$ associated
	to $(Y,d)$ is defined as follows:
	\begin{itemize}
		\item for any $r\geq0$, we define $E_{r}=\{(y,z)\in Y\times Y:d(y,z)\leq r\}$; 
		\item as a topological space, we have $\CG_{(Y,d)}=\bigcup_{r\geq0}\overline{E_{r}}$
		inside $\beta(Y\times Y)$, Stone-\v{C}ech compactification of $Y\times Y$; 
		\item we have $\CG_{(Y,d)}^{(0)}=\overline{E_{0}}\cong\beta Y$;
		\item the range and source maps are, respectively, the unique extensions
		of the first and second factor maps $Y\times Y\to Y$;
		\item the multiplication is the unique extension of the composition map
		$\CG_{(Y,d)}^{(2)}\cap((Y\times Y)\times(Y\times Y))\to(Y\times Y)$,
		$((y,z),(z,w))\mapsto(y,z)\circ(z,w)=(y,w)$, as we notice that $E_{r}\circ E_{s}\subseteq E_{r+s}$
		for any $r,s\geq0$. 
	\end{itemize}
	The uniform local finiteness of $(Y,d)$ implies that this groupoid
	is locally compact, Hausdorff, principal and \'{e}tale (c.f., \cite[Proposition~3.2]{SkandalisTuYu2002coarse}). It is also $\sigma$-compact by definition. \end{defn}
\begin{rem}
	\label{rem:coarse-groupoid-length}There is a canonical length function
	$\ell$ on $\CG_{(Y,d)}$ defined by extending the metric $d:Y\times Y\to[0,\infty]$
	to $\beta(Y\times Y)$ and observing that it takes finite values on
	$\CG_{(Y,d)}$. This length function is continuous by definition.
	It is also proper since by the density of $Y\times Y$ inside $\beta(Y\times Y)$
	, we have, for any $r>0$, the open set $\ell^{-1}([0,r))$ is contained
	in $\overline{\ell^{-1}([0,r))\cap(Y\times Y)}$, which in turn is
	contained in the compact set $\overline{E_{r}}$. Moreover, observe
	that for any $y\in Y$ viewed as a unit of $\CG_{(Y,d)}$, the source
	fiber $\left(\CG_{(Y,d)}\right)_{y}$, under the invariant fiberwise
	extended metric induced by $\ell$, is isometric to $(Y,d)$. \end{rem}
\begin{prop}
	\label{prop:coarse-groupoid-amenable}Let $(Y,d)$ be a uniformly
	locally finite extended metric space. Then it is amenable (respectively,
	ubiquitously amenable) if and only if $\CG_{(Y,d)}$ is fiberwise
	amenable (respectively, ubiquitously fiberwise amenable). \end{prop}
\begin{proof}
	Let $\ell$ be the canonical length function on $\CG_{(Y,d)}$ given
	in Remark~\ref{rem:coarse-groupoid-length} and let $\rho$ be the
	induced invariant fiberwise extended metric. By Proposition~\ref{prop:fiberwise-amenability-metric},
	it suffices to show $(Y,d)$ is amenable (respectively, ubiquitously
	amenable) if and only if $\left(\CG_{(Y,d)},\rho\right)$ is. We observed
	that $(Y,d)$ is isometric to $\left(\left(\CG_{(Y,d)}\right)_{y},\rho\right)$
	for any $y\in Y$; thus $(Y,d)$ embeds isometrically into $\left(\CG_{(Y,d)},\rho\right)$
	as some of the coarse connected components. It follows that the amenability
	of $(Y,d)$ implies that of $\left(\CG_{(Y,d)},\rho\right)$, while
	the ubiquitous amenability of the latter implies that of the former. 
	
	Now we assume $\left(\CG_{(Y,d)},\rho\right)$ is amenable and show
	so is $(Y,d)$. Given $R,\varepsilon>0$, we apply Lemma~\ref{lem:Folner-components}
	to obtain a unit $u\in\beta Y$ and an $(R,\varepsilon)$-F{\o}lner set
	$F$ in $\left(\CG_{(Y,d)}\right)_{u}$, and then apply Lemma~\ref{lem:Folner-permeates}
	together with the density of $Y$ in $\beta Y$ to obtain $y\in Y$
	and an $(R,\varepsilon)$-F{\o}lner set $F'$ in $\left(\CG_{(Y,d)}\right)_{y}$.
	Since each $\left(\CG_{(Y,d)}\right)_{y}$ is isometric to $(Y,d)$
	and $R$ and $\varepsilon$ were chosen arbitrarily, this shows $(Y,d)$
	is amenable. 
	
	Finally we assume $(Y,d)$ is ubiquitously amenable and show so is
	$\left(\CG_{(Y,d)},\rho\right)$. Thus given $R,\varepsilon>0$, there
	exists $S>0$ such that for any $y\in Y$, there exists an $(R,\varepsilon)$-F{\o}lner
	set in $\bar{B}_{d}(y,S)$. Now given $x\in\CG_{(Y,d)}$, we claim
	there exists an $(R,\varepsilon)$-F{\o}lner set $F$ in $\bar{B}_{\rho}(x,S+\varepsilon)$.
	Indeed, if $s(x)\in Y$, then this follows from the fact that $\left(\CG_{(Y,d)}\right)_{s(x)}$
	is isometric to $(Y,d)$. If $s(x)\in\beta Y\setminus Y$ instead,
	then by Lemma~\ref{lem:Folner-permeates}, there is an open neighborhood
	$V$ of $r(x)$ in $\CG_{(Y,d)}^{(0)}$ such that the existence of
	an $(R,\eps)$-F{\o}lner set $F$ in $\bar{B}_{\rho}\left(r(x),S+\varepsilon\right)$
	is implied by the existence of an $(R,\eps)$-F{\o}lner set in $\bar{B}_{\rho}\left(v_{0},S\right)$
	for some $v_{0}\in V$, but the latter condition holds as soon as
	we pick $v_{0}\in Y\cap V$ by the density of $Y$ in $\beta Y$.
	It follows that $Fx$ is an $(R,\eps)$-F{\o}lner set in $\bar{B}_{\rho}\left(x,S+\varepsilon\right)$,
	as desired. 
\end{proof}
We find it intriguing that while Lemma~\ref{lem:Folner-permeates}
(which depends on the local slice lemma) is used in the above proof
for both amenability and ubiquitous amenability, the two instances
occur in opposite directions. 

\begin{rem}\label{rem:fa-vs-ta}
    It can be seen in the case of coarse groupoids that fiberwise amenability is neither stronger nor weaker than topological amenability, since the latter corresponds to Yu's property A on a metric space \cite[Theorem~5.3]{SkandalisTuYu2002coarse}, but there are easy examples that demonstrate metric amenability and property A do not cover each other (see, e.g., \cite[Remark~2.7]{AraLiLledoWu2018Amenabilitya}). 
\end{rem}

\subsection{Almost finite and purely infinite groupoids}\label{subsec:af-pi}

A major motivation that prompted us to introduce and study fiberwise amenability comes from a pair of properties that have played significant roles in the recent development of the theory of \'{e}tale groupoids\textemdash almost finitenss and pure infiniteness.  

\begin{defn}[{\cite[Definition 6.2]{Matui2012Homology}}]
	\label{6.1} \label{defn:Matui}
	An ample \'{e}tale
	groupoid $\CG$ with a compact unit space is called \textit{almost
		finite} if, for any compact set $K$ in $\CG$ and $\epsilon>0$, there
	is a compact open elementary subgroupoid $\CH$ of $\CG$ with $\HU=\GU$
	such that for any $u\in\GU$, one has 
	\[
	|K\CH u|<(1+\epsilon)|\CH u|
	\; .
	\]
\end{defn}

Comparing \Cref{defn:Matui} with \Cref{def:fiberwise-amenable}\eqref{def:fiberwise-amenable:ubiquitous}, we arrive at the following observation. 

\begin{prop}\label{prop:af-ufa}
    If an ample \'{e}tale groupoid $\CG$ is almost finite, then it is ubiquitously fiberwise amenable. 
\end{prop}

The other property in focus here, pure infiniteness, was also introduced by Matui \cite[Definition~4.9]{Matui2015Topological}: 
an ample \'{e}tale groupoid is \emph{purely infinite} if every clopen subset $C \subseteq \CG^{(0)}$ is the common source of two clopen bisections with disjoint ranges inside $C$. It is clear from the definition that a purely infinite groupoid admits no invariant probability measures on $\CG^{(0)}$; in particular, it cannot be fiberwise amenable or almost finite. 

Both almost finiteness and pure infiniteness may be regarded as regularity properties for groupoids, in that they each guarantee a number of nice results about groupoids hold (see, e.g., \cite{Matui2015Topological,Matui2016Etale}), even though the two properties are mutually exclusive. In a sequel \cite{MWAE}, we introduce the notion of \emph{almost elementariness} as a natural common generalization of almost finiteness and pure infiniteness, with fiberwise amenability serving as the criterion that distinguishes the two.

The following theorem is designed for application in \cite{MWAE}. It showcases
a  F\o{}lner--versus--paradoxical type result for $\sigma$-compact
\'{e}tale groupoids on compact spaces. 
In cases that interest us the most (e.g., minimal groupoids, for which we have \Cref{thm:minimal-fiberwise-amenability}), we have a genuine dichotomy. 
\begin{thm}
\label{4.11} 
\label{thm:fiberwise-amenable-dichotomy}
Let $\CG$ be a 
$\sigma$-compact \'{e}tale groupoid with a compact unit space. 
Then we have the following. 
\begin{enumerate}
\item \label{thm:fiberwise-amenable-dichotomy:amenable} If $\CG$ is ubiquitous fiberwise amenable, then 
for any compact subset $K$ in $\CG$ and any $\varepsilon>0$, 
there is a compact set $L \subseteq \CG$ such that
for any finite set $M \subset\CG$, there
is a finite set $F$ satisfying 
\[
	M \subseteq F\subseteq LM \ \textrm{and}\ |KF|\leq(1+\varepsilon)|F|.
\]

\item \label{thm:fiberwise-amenable-dichotomy:non-amenable} If $\CG$ is not fiberwise amenable, then 
for any compact set $K \subseteq\CG$
and any $n\in\mathbb{N}$, there is a compact set $L \subseteq\CG$ such that
for any finite set $M\subseteq\CG$, the set
$LM$ contains at least $n|M|$ many disjoint sets of the form $K x$ for $x \in \CG$. 
\end{enumerate}
\end{thm}
\begin{proof}
Let $\ell$ be a coarse continuous length function and $\rho$ the induced canonical metric. 
Observe that by enlarging $K$ if necessary, we may assume without loss of generality that $K = \ell^{-1} ( [0,r] )$ for some $r \geq 0$ in both statements above. 
Then \eqref{thm:fiberwise-amenable-dichotomy:amenable} and~\eqref{thm:fiberwise-amenable-dichotomy:non-amenable}, respectively, follow from converting Propositions~\ref{3.5} and~\ref{3.7}, respectively, into the groupoid setting using Lemma~\ref{lem:invariant-fiberwise-extended-metric} and then taking $L := \ell^{-1} ( [0, S])$. 
\end{proof}

To end this section, we record an example due to G\'{a}bor Elek. This
example indicates that our ubiquitous fiberwise amenability in general
not necessarily implies (topologically) amenability of groupoids.
However, for transformation groupoids, it is a well-known
fact that any action of an amenable group is (topologically) amenable.

\begin{example}
	\label{exa: Elek} In \cite[Theorem 6]{Elek-qualitative}, Elek constructed
	a class of groupoids, called \textit{geometric groupoid} by using
	so-called stable actions. They are (second
	countable) minimal principal almost finite ample \'{e}tale groupoids but
	not (topologically) amenable. However, it follows from \Cref{prop:af-ufa} that Elek's geometric groupoids are ubiquitously fiberwise amenable. 
\end{example}

\section*{Acknowledgements}
This research is supported by National Key R\&D Program of China 2022YFA100700. 
The first author is supported by NSFC No.~12571133. 
The second author is partially supported by NSFC Key Program No.~12231005. The authors would like to thank the anonymous referee for the very helpful comments and suggestions which improve the paper significantly.

\bibliographystyle{alpha}
\bibliography{Almost_elementary}

\end{document}